\documentclass{amsart}

\usepackage[dvipdfmx]{graphicx}
\usepackage{amsmath}
\usepackage{amssymb}
\usepackage{amsthm}
\usepackage{color}

\usepackage{tabmacD}
\usepackage[all]{xy}

\title[Jeu de taquin, rectification, and ultradiscrete KP]{Jeu de taquin, uniqueness of rectification,\\
and ultradiscrete KP}
\author{Shinsuke Iwao}
\address{Department of Mathematics, Tokai University, 4-1-1, Kitakaname, Hiratsuka, Kanagawa 259-1292, Japan.}
\email{iwao@tokai.ac.jp}
\date{\today}

\newtheorem{thm}{Theorem}[section]
\newtheorem{prop}[thm]{Proposition}
\newtheorem{lemma}[thm]{Lemma}
\newtheorem{defi}[thm]{Definition}
\newtheorem{example}[thm]{Example}
\newtheorem{rem}[thm]{Remark}

\newtheorem{cor}[thm]{Corollary}

\def\RR{\mathord{\mathbb{R}}}
\def\ZZ{\mathord{\mathbb{Z}}}
\def\kasane#1#2{\genfrac{}{}{0pt}{2}{#1}{#2}}
\def\la{\leftarrow}
\renewcommand\tilde[1]{\widetilde{#1}}

\newcommand{\Tableau}[1]{%
\def\emp{\emptyset\bl}
{\text{\tableau[sY]{#1}}}%
}%

\newcommand{\bm}[1]{\mbox{\boldmath{$#1$}}}

\newcommand{\ktableau}[1]{%
\def\k{\lower-1.4ex\hbox{$k$}}
\def\b{\lower1ex\hbox{$\circ$}}
\tableau{#1}%
}

\makeatletter

\def\h@iti#1{%
\bgroup%
\let\\=\cr%
\def\r{\hbox to 0pt{$\to$}}%
\def\d{\vbox to 0pt{\hbox to 0pt{$\searrow$}}}%
\vcenter{\tabskip=0pt\halign{&$##$\hspace{10pt}\cr#1\crcr}}%
\egroup%
}

\newcommand{\haiti}[1]{\hspace{4.5pt}\h@iti{#1}\hspace{-4.5pt}}

\makeatother

\def\sayou#1{%
\bgroup%
\let\\=\cr%
\def\+{\hbox{\lower 5pt\hbox{\scalebox{3.5}{$+$}}}}%
\def\-{\hbox{\lower 5pt\hbox{\scalebox{3.5}{$+$}\lower -2pt\hbox to 0pt{\hspace{-19.9pt}\scalebox{2.5}{$\circ$}}}}}%
\vcenter{\tabskip=0pt\halign{&\hfil$##$\hfil\cr#1\crcr}}%
\egroup%
}

\def\o#1{\overline{#1}}

\begin{document}

\maketitle

\begin{abstract}
In this paper, we study tropical-theoretic aspects of the ``rectification algorithm'' on skew Young tableaux.
It is shown that the algorithm is interpreted as a time evolution of some tropical integrable system.
By using this fact, we construct a new combinatorial map that is essentially equivalent to the rectification algorithm.
Some of properties of the rectification can be seen more clearly via this map.
For example, the uniqueness of a rectification boils down to an easy combinatorial problem.
Our method is mainly based on the two previous researches: the theory of geometric tableaux by Noumi-Yamada, and the study on the relationship between jeu de taquin slides and the ultradiscrete KP equation by Mikami and Katayama-Kakei.
\end{abstract}

\section{Introduction}\label{sec:1}

The {\it tropicalization} is a procedure to translate mathematical statements (such as propositions, equations, formulas, \textit{etc}.) 
written in the ``language of rings'' into the ``language of semi-fields,'' where the addition, the multiplication, and the multiplicative inverse are transformed as
\[
a+b\mapsto \min[A,B],\qquad ab\mapsto A+B,\qquad a^{-1}\mapsto -A.
\]
For example, rational maps correspond to piecewise linear maps.

In 2001, Berenstein and Kirillov~\cite{bernsteinRSK2001} showed that the Robinson-Schensted-Knuth correspondence (RSK correspondence), a crucial bijection in the theory of Young tableaux, can be expressed as a piecewise linear map that is related with the crystal basis.
After that, Kirillov~\cite{2001phco.conf...82K} introduced the {\it geometric RSK correspondence} (originally, tropical RSK correspondence\footnote{The word ``tropical'' nowadays has a different meaning.
Many researchers seem to prefer to use the ``geometric RSK correspondence'' instead.}), which is a rational map obtained by ``lifting'' Berenstein-Krillov's piecewise linear map.
This correspondence was studied further by Noumi-Yamada~\cite{noumi2004} in terms of tropical integrable systems.
In these literature, the method of tropicalization plays a key role in deriving explicit formulas.
Such techniques were originally refereed to as the ``tropical approach~\cite{2001phco.conf...82K,noumi2004}.''

Recently, Mikami~\cite{mikami2012en} and Katayama-Kakei~\cite{kakei2015en} found an interesting relationship between \textit{jeu de taquin slides} and the \textit{tropical KP equation}.
What is interesting is the fact that their correspondence is (probably) independent of Noumi-Yamada's correspondence.
For this reason, Young tableaux are viewed as a significant example that admits (at least) two independent realizations by tropical integrable systems\footnote{The {\it Takahashi-Satsuma Box-Ball system}~\cite{takahashi1990soliton} is a good example that admits two tropical realizations: one is the \textit{ultradiscrete Toda equation}, and the other is the \textit{ultradiscrete KdV equation}. 
}.

By using these facts, we present a new approach for solving problems on combinatorics of Young tableaux.
In this paper, we introduce a new combinatorial map that is essentially equivalent to the \textit{rectification algorithm}~\cite{fulton_1996} in terms of \textit{tropical}(=ultradiscrete) \textit{relativistic Toda equation}~\cite{iwao2018discrete}.
This tropical-theoretic realization provides a different interpretation, where some properties of the algorithm can be seen clearly.
For instance, through this map, a proof of {\it the uniqueness of a rectification}, which is probably first nontrivial theorem in the theory of Young tableaux, boils down to an easy combinatorial problem.

For another application, we would mention the author's related work~\cite{iwao2018Tropical} on the \textit{shape equivalence} and the \textit{Littlewood-Richardson correspondence}.
We expect further researches will reveal deeper relationships between combinatorics and tropical integrable systems.

\subsubsection*{Contents of the paper}

This work is inspired by the recent studies of Mikami~\cite{mikami2012en} and Katayama-Kakei~\cite{kakei2015en}(see \S\ref{sec:2}, Theorem \ref{thm:kakei}) on relationships between {jeu de taquin slides} (cf.\,\S \ref{sec:appB}) and the ultradiscrete KP equation (\ref{eq:udKP})\footnote{In the paper \cite{kakei2015en}, they considered standard skew tableaux only.
However, their proof is valid for general skew tableaux without any changes.}.

Under a change of variables, the ultradiscrete KP equation (\ref{eq:udKP}) is transformed to the recursive form (\ref{eq:evolution}), which is more suitable for our study.
We give a proof of this fact by the method of \textit{tropicalization} (\S \ref{sec:3}).
What should be noted here is that {\it the formal tropicalization of a true proposition is not always true}.
For example, ``$a+b=a+c$ implies $b=c$'' does not mean ``$\min[A,B]=\min[A,C]$ implies $B=C$.''
In many cases, such problems are easily avoided by direct methods --- for example, by simplifying expressions --- but it sometimes cause errors that are difficult to find.
In \S \ref{sec:3.2}, we present a systematic approach to such problems by means of \textit{mathematical logic}.
As an application, a formal proof of (\ref{eq:evolution}) is given in \S \ref{sec:3.3}.

In \S \ref{sec:4}, we introduce a diagrammatic algorithm for calculating the time evolution of (\ref{eq:evolution}).
This algorithm also provides a good reinterpretation of jeu de taquin slides. 
(See the example given in \S \ref{sec:4.3}.)

In \S \ref{sec:5} and \S \ref{sec:6}, we give an alternative proof of the ``uniqueness of a rectification~\cite[\S 1, Claim 2]{fulton_1996},'' which is probably the first non-trivial theorem in the course of combinatorics of Young tableaux.
The key of the proof is a commutative diagram 
\begin{equation*}
\def\myspace{4em}%
\xymatrix{
\hspace{\myspace}\{Q,W\}\hspace{\myspace}\ar@{|->}[r]^{\mathrm{Jeu\ de\ taquin} (\S \ref{sec:4},\S \ref{sec:appA.2})} 
\ar@{|->}[d]_{
\substack{
\mathrm{Row\ insertion}\\
(\simeq \mathrm{Geometric\, tableau})\\
(\S \ref{sec:5.2},\S \ref{sec:appA.2})
}
} 
& \hspace{\myspace}\{Q',W'\}\hspace{\myspace}\ar@{|->}[d]^{
\mathrm{Row\ insertion}
}\\
\hspace{\myspace}\{P,W\}\hspace{\myspace}\ar@{|->}[r]^{
\substack{
\mathrm{A\ new\ map\ constructed}\\
\mathrm{via\ tropical\ integrable\ systems}\\
(\S \ref{sec:5.4})
}
} & \hspace{\myspace}\{P',W'\}\hspace{\myspace}
},
\end{equation*}
where $W$ is a skew tableau (\S \ref{sec:4.1}), $W'$ is a rectification of $W$ (\S \ref{sec:5.1}), $Q$ is a sequence of row numbers from where jeu de taquin slides start (\S \ref{sec:5.1}), and $P$ is the \textit{$P$-tableau} (\S \ref{sec:5.3}) associated with $Q$.
See \S \ref{sec:5.5}.
From this diagram, we find that the rectification depends only on the choice of a $P$-tableau.
Then the uniqueness of a rectification boils down to a relatively easy lemma (Corollary \ref{cor:U(mu)}) that states a $P$-tableau of given shape is unique.

In the appendix, we give a short list of basic definitions in mathematical logic in \S \ref{sec:appA}. 
A brief introduction to combinatorics of Young tableaux is given in \S \ref{sec:appB}.

\subsubsection*{Notations}

In this article, we follow the convention of Fulton's textbook~\cite{fulton_1996}.
Let $\lambda=(\lambda_1\geq \lambda_2\geq \dots\geq \lambda_\ell>0)$ be a Young diagram.
A {\it semi-standard tableau of shape $\lambda$} is obtained by filling the boxes in $\lambda$ with a number according to the following rules:
(i) in each row, the numbers are weakly increasing from left to right,
(ii) in each column, the numbers are strongly increasing from top to bottom.
A semi-standard tableau is often referred to as a {\it tableau} shortly.
A tableau with $n$ boxes is called {\it standard} if it contains distinct $n$ numbers $1,2,\dots,n$.
Let $\lambda/\mu$ be a skew diagram, where $\lambda$ and $\mu$ are Young diagrams with $\mu\subset \lambda$.
A {\it skew (semi-standard) tableau of shape $\lambda/\mu$} is obtained by filling the boxes with a number according to the same rules as for tableaux.
If a skew tableau with $n$ boxes contains distinct $n$ numbers $1,\dots,n$, it is said to be {\it standard}.
See \S\ref{sec:appB} for other definitions.

\subsection*{About this article}
Most part of this paper (except for \S \ref{sec:5.4}, \S \ref{sec:5.5}, and \S \ref{sec:6}) is an English translation of the author's unpublished manuscript ``S.~Iwao, \textit{Jeu de taquin, uniqueness of a rectification, and ultradiscrete KP}'' written in Japanese\footnote{
The original manuscript (in Japanese) had been submitted to \textit{RIMS K\^{o}ky\^{u}roku Bessatsu}, but was withdrawn (03/07/2019) because the first version of the manuscript had contained a major mathematical error.}.

\section{Ultradiscrete (tropical) KP equation and jeu de taquin}\label{sec:2}

In this section, we introduce the result of Katayama and Kakei~\cite{kakei2015en} in 2015.
The definition of the terms {\it jeu de taquin slide, inside corner, outside corner, etc.}~can be found in \S \ref{sec:appB}.

Let us consider the discrete KP equation
\begin{equation}\label{eq:discreteKP}
f^t_{i+1,j+1}f^{t+1}_{i-1,j}-f^t_{i,j}f^{t+1}_{i,j+1}+f^t_{i,j+1}f^{t+1}_{i,j}=0.
\end{equation}
According to the definition of tropicalization introduced in \S \ref{sec:1}, the ``tropicalization of (\ref{eq:discreteKP})'' should be the following piecewise linear equation:
\begin{equation}\label{eq:udKP}
F_{i,j}^t+F_{i,j+1}^{t+1}=\max\left[
F_{i+1,j+1}^t+F_{i-1,j}^{t+1},
F_{i,j+1}^t+F_{i,j}^{t+1}
\right].
\end{equation}
The following is the main theorem of \cite{kakei2015en}:
\begin{thm}[\cite{kakei2015en}. (See also \cite{mikami2012en})]\label{thm:kakei}
Let $\{S^t\}_{t=0,1,2,\dots}$ be a collection of skew tableaux such that $S^{t+1}$ is obtained from $S^t$ by a jeu de taquin slide, carried out from any inside corner.
Define
\[
F_{i,j}^t=
\sharp
\left(
\mbox{
$0,1,2,\dots,j$'s contained in the $1^\mathrm{st},2^\mathrm{nd},\dots,i^{\mathrm{th}}$ rows of $S^t$ }
\right),
\]
where an empty box is regarded as a box with $0$.
Then $\{F_{i,j}^t\}_{i\geq 1,j\geq 0,t\geq 0}$ satisfies the ultradiscrete KP equation $(\ref{eq:udKP})$.
\end{thm}

\begin{example}\label{example:first}
Consider the sequence of jeu de taquin slides displayed below.
The gray boxes denote the inside corners from which a jeu de taquin slide is carried out.
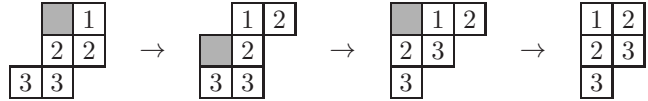
\[
\Tableau{\bl & \gray & 1\\
\bl & 2 & 2\\
3 & 3}\quad\to\quad
\Tableau{\bl & 1 & 2\\
\gray & 2 \\
3 & 3}\quad\to\quad
\Tableau{\gray & 1 & 2\\
2 & 3 \\
3}\quad\to\quad
\Tableau{1 & 2 \\
2 & 3 \\
3}
\]
Let $S^0,S^1,S^2,S^3$ denote these skew tableaux.
The first $3\times 4$ part of the matrix $F^t=(F_{i,j}^t)_{\kasane{i\geq 1}{j\geq 0}}$ is expressed as
\[
F^0=\left(
\haiti{
2 & 3 & 3 & 3 \\
3 & 4 & 6 & 6 \\
3 & 4 & 6 & 8 
}
\right),\
F^1=\left(
\haiti{
1 & 2 & 3 & 3 \\
2 & 3 & 5 & 6 \\
2 & 3 & 5 & 8 
}
\right),
\]
\[
F^2=\left(
\haiti{
1 & 2 & 3 & 3 \\
1 & 2 & 4 & 5 \\
1 & 2 & 4 & 6 
}
\right),\
F^3=\left(
\haiti{
0 & 1 & 2 & 2 \\
0 & 1 & 3 & 4 \\
0 & 1 & 3 & 5 
}
\right).
\]
\end{example}

Putting 
\begin{equation}\label{eq:transform}
I_{i,j}^t=\frac{f_{i-1,j}^{t}f_{i,j}^{t+1}}{f_{i,j}^{t}f_{i-1,j}^{t+1}},
\qquad
V_{i,j}^t=\frac{f_{i-1,j}^{t}f_{i+1,j+1}^{t}}{f_{i,j}^{t}f_{i,j+1}^{t}},
\end{equation}
we rewrite the discrete KP equation (\ref{eq:discreteKP}) as
\begin{equation}\label{eq:discreteToda}
\begin{cases}
I_{i,j}^tV_{i,j}^{t+1}=I_{i+1,j+1}^tV_{i,j}^{t},\\
I_{i,j}^t-V_{i-1,j}^{t+1}=I_{i,j+1}^t-V_{i,j}^{t}.
\end{cases}
\end{equation}
Moreover, by putting
\[
R_j^t:=\left(
\begin{array}{ccccc}
I_{1,j}^t & 1 &  &   \\ 
 & I_{2,j}^t & 1 &  \\ 
 &  & I_{3,j}^t & \ddots   \\ 
 &  &  & \ddots 
\end{array} 
\right),\
L_j^t:=\left(
\begin{array}{ccccc}
1 &  &  &   \\ 
-V_{1,j}^t & 1 & &  \\ 
 & -V_{2,j}^t & 1 &    \\ 
 &  & -V_{3,j}^t & 1 \\ 
 &  &  & \ddots &\ddots
\end{array} 
\right)^{-1},
\]
(\ref{eq:discreteToda}) is transformed into the matrix form:
\begin{equation}\label{eq:Laxform}
R_j^tL_j^t=L_j^{t+1}R_{j+1}^{t}.
\end{equation}
Equation (\ref{eq:Laxform}) is often refereed to as the {\it discrete $(2+1)$-dimensional Toda equation}. 
It is easily checked that (\ref{eq:discreteToda}) is equivalent to the subtraction-free form:
\begin{equation}\label{eq:subtractionfree}
I_{i+1,j+1}^t=\frac{I_{i+1,j}^t+V_{i+1,j}^t}{I_{i,j}^t+V_{i,j}^t}I_{i,j}^t,\qquad
V_{i,j}^{t+1}=\frac{I_{i+1,j}^t+V_{i+1,j}^t}{I_{i,j}^t+V_{i,j}^t}V_{i,j}^t.
\end{equation}

We now ``tropicalize'' (\ref{eq:transform}) and (\ref{eq:subtractionfree}).
Let $Q_{i,j}^t$ and $W_{i,j}^t$ be the tropicalization of $I_{i,j}^t$ and $V_{i,j}^t$, respectively.
Then the tropicalization of (\ref{eq:transform}) is
\begin{equation}\label{eq:defofQW}
\begin{aligned}
&Q_{i,j}^t=F_{i,j}^t+F_{i-1,j}^{t+1}-F_{i-1,j}^t-F_{i,j}^{t+1},\\
&W_{i,j}^t=F_{i,j}^t+F_{i,j+1}^{t}-F_{i-1,j}^t-F_{i+1,j+1}^{t},
\end{aligned}
\end{equation}
and that of (\ref{eq:subtractionfree}) is
\begin{equation}\label{eq:evolution}
\begin{aligned}
&Q_{i+1,j+1}^t=(\min[Q_{i+1,j}^t,W_{i+1,j}^t]-\min[Q_{i,j}^t,W_{i,j}^t])+Q_{i,j}^t,\\
&W_{i,j}^{t+1}=(\min[Q_{i+1,j}^t,W_{i+1,j}^t]-\min[Q_{i,j}^t,W_{i,j}^t])+W_{i,j}^t.
\end{aligned}
\end{equation}

On the analogy of ``\{(\ref{eq:discreteKP}) and (\ref{eq:transform})\} $\Rightarrow$ (\ref{eq:subtractionfree}),''  it would be natural to expect that the implication ``\{(\ref{eq:udKP}) and (\ref{eq:defofQW})\} $\Rightarrow$ (\ref{eq:evolution})'' is true.
Note, however, that it is not obvious at this stage. (See \S \ref{sec:1}.)

\begin{example}
For the skew tableaux in Example \ref{example:first}, we have
\[
Q^0=\left(
\haiti{
1 & 1 & 0 & 0 \\
0 & 0 & 1 & 1 \\
0 & 0 & 0 & 0
}
\right),\ 
Q^1=\left(
\haiti{
0 & 0 & 0 & 0 \\
1 & 1 & 1 & 0 \\
0 & 0 & 0 & 1
}
\right),\ 
Q^2=\left(
\haiti{
1 & 1 & 1 & 1 \\
0 & 0 & 0 & 0 \\
0 & 0 & 0 & 0
}
\right),
\]
\[
W^0=\left(
\haiti{
1 & 0 & 0 & 0 \\
1 & 1 & 1 & 1 \\
0 & 0 & 0 & 2
}
\right),\ 
W^1=\left(
\haiti{
0 & 0 & 1 & 1 \\
1 & 1 & 0 & 0 \\
0 & 0 & 0 & 2
}
\right),\ 
W^2=\left(
\haiti{
1 & 1 & 1 & 1 \\
0 & 0 & 0 & 1 \\
0 & 0 & 0 & 1
}
\right).
\]
It is directly checked that $(\ref{eq:evolution})$ is satisfied.
\end{example}

\section{Tropical approach}\label{sec:3}

As we have seen in the previous section, it would be natural to expect that propositions and theorems written in ``the language of rings'' imply the similar facts written in ``the language of semi-fields,'' while it is not generally true.
In this section, we propose a formal method to deal with such ideas systematically.

We review the ``naive'' principle of tropicalization in \S\ref{sec:3.1}, and introduce its formal counterpart in \S\ref{sec:3.2}.
As an application, we give a formal proof of ``\{(\ref{eq:udKP}) and (\ref{eq:defofQW})\} $\Rightarrow$ (\ref{eq:evolution})'' in \S\ref{sec:3.3}.
This section can be skipped if the reader is interested only in combinatorics.

\subsection{``Naive'' tropical approach}\label{sec:3.1}

A real polynomial $f(x)\in \RR[x_1,\dots,x_N]$ is called {\it subtraction-free} if it is expressed as $f(x)=\sum_{I}c_{I}x^{I}$, where $x^I=x_1^{i_1}\cdots x_N^{i_N}$ and $c_I\geq 0$ for all $I=(i_1,\dots,i_N)$．

\begin{defi}
Let $f(x)=\sum_{I}c_{I}x^{I}$ be a subtraction-free polynomial.
The {\it tropicalization of $f(x)$} is the piecewise linear function $\overline{f}(X_1,\dots,X_n)$ defined as 
$$\overline{f}(X_1,\dots,X_N)=\min_{\{I:c_I\neq 0\}}[I_1\cdot X_1+\cdots+I_N\cdot X_N],$$
where $\min\emptyset=+\infty$.
\end{defi}

\begin{prop}[``Naive'' principle of tropicalization]\label{principle:trop}
Let $f(x_1,\dots,x_n)$ be a subtraction-free polynomial.
By assuming $x_i= O(e^{-\frac{X_i}{\epsilon}})$ $(\epsilon \downarrow 0)$, where $\epsilon$ is a positive parameter, we have
$-\lim\limits_{\epsilon\downarrow 0}\epsilon \log f(x_1,\dots,x_n)=\overline{f}(X_1,\dots,X_n)$.
\end{prop}

\begin{example}\label{example:good-example}
Assume $a,b,c,d,e,f,g$ satisfy $a=b+c$, $d=e+g$, and $f=bg$, which imply, for example, $ad+ce=ae+cd+f$.
The tropicalization of this statement is 
\begin{align*}
(A=\min[B,C])\wedge(D=\min[E,G])\wedge (F=B+G)\ \\
\Rightarrow\ \min[A+D,C+E]=\min[A+E,C+D,F].
\end{align*}
We will give a proof of this proposition by the method of tropicalization.
Assume $A,B,\dots,F\in \RR$ satisfy the assumption of the implication.
Set $b(\epsilon)=e^{-\frac{B}{\epsilon}}$, $c(\epsilon)=e^{-\frac{C}{\epsilon}}$, $e(\epsilon)=e^{-\frac{E}{\epsilon}}$, and $g(\epsilon)=e^{-\frac{G}{\epsilon}}$, where $\epsilon$ is a positive parameter.
Let us define $a(\epsilon):=b(\epsilon)+c(\epsilon)$.
$d(\epsilon):=e(\epsilon)+g(\epsilon)$, and $f(\epsilon):=d(\epsilon)g(\epsilon)$.
From the principle of tropicalization (Proposition \ref{principle:trop}), the tropicalization of  $a(\epsilon),b(\epsilon),\dots,f(\epsilon)$ coincides with $A,B,\dots,F$, respectively.
Then the desired proposition follows from $a(\epsilon)d(\epsilon)+c(\epsilon)e(\epsilon)=a(\epsilon)e(\epsilon)+c(\epsilon)d(\epsilon)+f(\epsilon)$.
\end{example}

\begin{example}\label{ex:counter}
The subtraction-free equation $a+b=a+c$ implies $b=c$, but $\min[A,B]=\min[A,C]$ does not imply $B=C$.
Indeed, $(A,B,C)=(0,1,2)$ is a counterexample.
\end{example}

Example \ref{ex:counter} is a simplest and typical example where the tropicalization causes an error.

\subsection{Tropical approach in terms of first-order logic}\label{sec:3.2}

In this section, we rephrase the method of tropicalization in terms of first-order logic.
Basic definitions and notions of first-order logic are introduced in \S \ref{sec:appA}.
For readers who are interested in mathematical logic, we recommend the standard textbooks~\cite{marker2006model,tent2012course}.

Let 
$$
\mathcal{L}=\{f_1,f_2,\dots,R_1,R_2,\dots, c_1,c_2,\dots  \}
$$
be a language, where $f_i$ is a function symbol, $R_i$ is a relation symbol, and $c_i$ is a constant symbol.
Consider the two $\mathcal{L}$-structures:
\[
\begin{gathered}
\mathcal{M}=(M,f_1,f_2,\dots,R_1,R_2,\dots,c_1,c_2,\dots),\\
\o {\mathcal{M}}=(\overline{M},\o {f_1},\o{f_2},\dots,\o{R_1},\o{R_2},\dots,\o{c_1},\o{c_2},\dots)
\end{gathered}
\]
($M$ is the domain of $\mathcal{M}$, and $\o{M}$ is the domain of $\o{\mathcal{M}}$) and a homomorphism
\[
M\to \o{M};\quad x\mapsto \overline{x}
\]
of $\mathcal{L}$-structures.

We use the following lemma, which we will prove in the appendix:
\begin{lemma}\label{prop:inherit}
For any negation-free\footnote{See \S \ref{sec:appA}, Definitions \ref{defi:formula} and \ref{def:nag-free-formula}.} $\mathcal{L}$-formula $\psi(x_1,\dots,x_n)$ and any element 
$(a_1,\dots,a_n)$ of $M^n$,
\begin{center}
$\mathcal{M}\models \psi(a_1,\dots,a_n)$\quad implies\quad $\o{\mathcal{M}}\models \o{\psi}(\o{a_1},\dots,\o{a_n})$.
\end{center}
\end{lemma}
\begin{proof}
See \S \ref{sec:appA.2}.
\end{proof}

\begin{prop}\label{prop:axiom-trop-principle}
Let $\mathcal{L}$-formulas $\varphi(x_1,\dots,x_n)$ and $\psi(x_1,\dots,x_n)$ satisfy:
\begin{enumerate}
\def\labelenumi{$(\theenumi)$.}
\setcounter{enumi}{-1}
\item $\psi(x_1,\dots,x_n)$ is negation-free.
\item For any $(A_1,\dots,A_n)\in \o{M}$ with $\overline{\mathcal{M}}\models\o{\varphi}(A_1,\dots,A_n)$, there exists some $(a_1,\dots,a_n)\in M^n$ that satisfies $\overline{a_i}=A_i$ and $\mathcal{M}\models\varphi(a_1,\dots,a_n)$,
\item 
$\mathcal{M}\models\forall x_1\dots\forall x_n(\varphi(x_1,\dots,x_n)\to \psi(x_1,\dots,x_n))$.
\end{enumerate}
Then, it follows that
\[
\o{\mathcal{M}}\models \forall X_1\dots\forall X_n(\o{\varphi}(X_1,\dots,X_n)\to \o{\psi}(X_1,\dots,X_n)).
\]
\end{prop}
\begin{proof}
Assume $\o{\mathcal{M}}\models\o\varphi(A_1,\dots,A_n)$ for some $(A_1,\dots,A_n)\in \o{M}^n$.
From (1), there exists $(a_1,\dots,a_n)\in M^n$ such that  $\overline{a_i}=A_i$ and $\mathcal{M}\models\varphi(a_1,\dots,a_n)$.
From (2), we have $\mathcal{M}\models\psi(a_1,\dots,a_n)$.
As a result, from (0) and Lemma \ref{prop:inherit}, we have $\o{\mathcal{M}}\models \o{\psi}(A_1,\dots,A_n)$.
\end{proof}

\begin{example}\label{example:trop-structure}
The ``naive'' principle of tropicalization (Proposition \ref{principle:trop}) is a special case of Proposition \ref{prop:axiom-trop-principle}.
Let $\mathcal{L}=\{+,\cdot,{}^{-1},1\}$\footnote{One may add the constant symbol ``$0$'' to $\mathcal{L}$, but it is not mandatory. 
Since all $\mathcal{L}$-formulas that we will use in this manuscript do not contain $0$, we can simply omit it.}.
Here $+,\cdot$ are binary function symbols, ${}^{-1}$ is a unary function symbol, and $1$ is a constant symbol.
Define the two $\mathcal{L}$-structures $\mathcal{M}=(M,+,\cdot,{}^{-1},1)$, $\o{\mathcal{M}}=(\o{M},\o{+},\o{\cdot},\o{{}^{-1}},\o{1})$ as follows:
\begin{itemize}
\item $M$ is the set of germs at $\epsilon=0$ of continuous positive functions $f(\epsilon)$ $(\epsilon>0)$ that satisfy $-\lim\limits_{\epsilon\downarrow 0}\epsilon\log f(\epsilon)\in \RR$.
\item $+$ denotes the standard addition, $\cdot$ denotes that standard multiplication, and ${}^{-1}$ denotes the multiplicative inverse.
$1=1(\epsilon)$ is the constant function.
\item $\o{M}=\RR$.
\item $\o{+}=\min$, $\o{\cdot}=+$, $\o{{}^{-1}}=-$, $\o{1}=0$.
\item The map $M\to\o{M}$ is defined by $f(\epsilon)\mapsto -\lim\limits_{\epsilon\downarrow 0}\epsilon\log f(\epsilon)$.
(This map is usually called ``{\it ultradiscretization}''.)
\end{itemize}
``Subtraction-free polynomials'' is now simply rephrased as ``$\mathcal{L}$-terms.''
Note that there are other possible choices of $\mathcal{M}$.
For example, one can take $M$ as the set of real formal power series the lowest coefficient of which is positive, and $M\to \o{M}$ as the valuation map.
\end{example}

There exists a useful sufficient condition for (1) in Proposition \ref{prop:axiom-trop-principle}.
Assume (i) $M\to \o{M}$ is onto and (ii) $\varphi(x_1,\dots,x_n)$ is ``a collection of definitions of next variables,'' namely, there exist some $1\leq \alpha\leq n$ and $\mathcal{L}$-terms $f_i(x_1,\dots,x_{i-1})$ $(i=\alpha+1,\alpha+2,\dots,n)$ such that
\begin{gather}
\varphi(x_1,\dots,x_n)=
\left(
\bigwedge_{i=\alpha+1}^n\{x_{i}=f_i(x_1,\dots,x_{i-1})\}
\right),\label{eq:recursive}
\end{gather}
where $x_1,\dots,x_{i-1}$ are free variables of $f_i(x_1,\dots,x_{i-1})$.
Under the assumptions (i--ii), one can find an element $(a_1,\dots,a_n)\in M^n$ in (1) of Proposition~\ref{prop:axiom-trop-principle}. 
In fact, from (i), there exists an element $(a_1,\dots,a_\alpha)\in M^\alpha$ that satisfies $\o{a_1}=A_1,\dots,\o{a_\alpha}=A_\alpha$.
Putting $a_i:=f_i(a_1,\dots,a_{i-1})$ for $i=\alpha+1,\alpha+2,\dots,n$, we obtain the $n$-tuple $(a_1,\dots,a_n)$, which satisfies $\mathcal{M}\models\varphi(a_1,\dots,a_n)$.
Since $\o{\mathcal{M}}\models \o{\varphi}(A_1,\dots,A_n)$, the equation $\o{a_i}=\o{f_i(a_1,\dots,a_{i-1})}=A_i$ holds for $i>\alpha$.

\begin{defi}
We call an $\mathcal{L}$-formula $\varphi(x_1,\dots,x_n)$ of the form $(\ref{eq:recursive})$ {\it recursive}.
\end{defi}

\begin{rem}
The condition $($ii$)$ can be significantly generalized as
\[
\varphi(x_{j_1},\dots,x_{j_l})=\exists x_{i_1}\exists x_{i_2}\dots \exists x_{i_k}\tilde{\varphi}(x_1,\dots,x_n),\qquad
\tilde{\varphi}(x_1,\dots,x_n): \mbox{ recursive},
\]
where $1\leq i_1<i_2<\dots<i_k\leq n$ and $\{i_1,\dots,i_k\}\sqcup \{j_1,\dots,j_l\}=\{1,\dots,n\}$.
In fact, for any $(A_{j_1},\dots,A_{j_l})\in \o{M}^l$ with $\overline{\mathcal{M}}\models\o{\varphi}(A_{j_1},\dots,A_{j_l})$, there exists an $n$-tuple $(A_1,\dots,A_n)\in \o{M}^n$ that satisfies $\overline{\mathcal{M}}\models\o{\tilde{\varphi}}(A_1,\dots,A_n)$.
In a similar way to the quantifier-free case, one can find the desired $(a_1,\dots,a_n)\in M^n$.
\end{rem}

\begin{example}
Example \ref{example:good-example} follows from Proposition \ref{prop:axiom-trop-principle} if we put
\[
\begin{gathered}
\varphi(a,\dots,f)=((a=b+c)\wedge(d=e+g)\wedge(f=bg)),\\
\psi(a,\dots,f)=(ad+ce=ae+cd+f).
\end{gathered}
\]
\end{example}
\begin{example}
Consider the proposition ``for any $a,b>0$, the inequality $a^2-4b>0$ implies the existence of $x>0$ that satisfies $x^2-ax+b=0$.''
This proposition can be rewritten in terms of the language $\mathcal{L}$ as
\[
\mathcal{M}\models \forall a\forall b(\varphi(a,b)\to \psi(a,b)),
\]
where
$\varphi(a,b)=\exists y(a^2=4b+y)$ and $\psi(a,b)=\exists x(x^2+b=ax)$.
By Proposition \ref{prop:axiom-trop-principle}, we have
\[
\o{\mathcal{M}}\models \forall A\forall B
(
\exists Y(2A=\min[B,Y])\to \exists X(\min[2X,B]=A+X)
).
\]
\end{example}

\subsection{Proof of (\ref{eq:evolution}) by formal arguments}\label{sec:3.3}

We now give a proof of (\ref{eq:evolution}) by formal arguments that we have seen in the previous section.
Hereafter, we fix the language $\mathcal{L}$, and the $\mathcal{L}$-structures $\mathcal{M}$ and  $\o{\mathcal{M}}$ as Example \ref{example:trop-structure}.
Let 
\[
\varphi(f_{i,j}^t,f_{i,j+1}^{t+1},f_{i+1,j+1}^t,f_{i-1,j}^{t+1},f_{i,j+1}^t,f_{i,j}^{t+1})
\]
be the $\mathcal{L}$-formula
\[
f^t_{i,j}=(f^{t+1}_{i,j+1})^{-1}\cdot (f^t_{i+1,j+1}f^{t+1}_{i-1,j}+f^t_{i,j+1}f^{t+1}_{i,j}),
\]
which is equivalent to the discrete KP equation (\ref{eq:discreteKP}).
We write it as $\varphi_{i,j}^t$ in short.
We also write the first equation of (\ref{eq:transform}) as 
\[
\theta(I_{i,j}^t,f_{i-1,j}^t,f_{i,j}^{t+1},f_{i,j}^{t},f_{i-1,j}^{t+1})\qquad (\theta_{i,j}^t, \mbox{ in short}),
\] 
and the second equation as 
\[
\chi(V_{i,j}^t,f_{i-1,j}^{t},f_{i+1,j+1}^{t},f_{i,j}^{t},f_{i,j+1}^{t})\qquad (\chi_{i,j}^t,\mbox{ in short}).
\]
Further, let 
\[
\Phi_{i,j}^t=\Phi(I_{i,j}^t,I_{i+1,j}^{t},I_{i,j}^{t+1},V_{i,j}^{t},V_{i+1,j}^{t},V_{i,j}^{t+1})
\] 
denote (\ref{eq:subtractionfree}).

The proposition ``\{(\ref{eq:discreteKP}) and (\ref{eq:transform})\} implies (\ref{eq:subtractionfree}),'' which can be checked by straightforward algebraic calculations, is now rephrased as
\[
\begin{aligned}
\mathcal{M}\models
&\forall I_{i,j}^t, 
\forall I_{i+1,j}^t, 
\forall I_{i+1,j+1}^t, 
\forall V_{i,j}^t,
\forall V_{i,j}^{t+1}, 
\forall V_{i+1,j}^t,
\forall f_{i-1,j}^t, 
\forall f_{i,j}^t, 
\forall f_{i+1,j}^t,\\
&\forall f_{i,j+1}^t, 
\forall f_{i+1,j+1}^t, 
\forall f_{i+2,j+1}^t, 
\forall f_{i-1,j}^{t+1}, 
\forall f_{i,j}^{t+1}, 
\forall f_{i+1,j}^{t+1}, 
\forall f_{i,j+1}^{t+1}, 
\forall f_{i+1,j+1}^{t+1}
\\
&\quad
(\varphi_{i+1,j}^t\wedge \varphi_{i,j}^t\wedge
\theta_{i,j}^t\wedge\theta_{i+1,j}^t\wedge\theta_{i+1,j+1}^t\wedge
\chi_{i,j}^{t}\wedge\chi_{i,j}^{t+1}\wedge\chi_{i+1,j}^{t}
)\to \Phi_{i,j}^t.
\end{aligned}
\]
By ordering the variables properly, we find the assumption of the implication is recursive.
In fact, it is enough to order them as (any $I_{\ast,\ast}^\ast,V_{\ast,\ast}^\ast$)$>f_{i+1,j}^t>f_{i,j}^t>$(any other $f_{\ast,\ast}^\ast$).
Therefore, from Proposition \ref{prop:axiom-trop-principle}, its tropicalization is also true.
Then, the statement ``\{(\ref{eq:udKP}) and (\ref{eq:defofQW})\} implies (\ref{eq:evolution})'' is true over $\o{\mathcal{M}}$.
%
%

\section{Combinatorial interpretation of (\ref{eq:evolution})}\label{sec:4}

We give a combinatorial interpretation of (\ref{eq:evolution}), which will help us to understand the relationship between (\ref{eq:evolution}) and the jeu de taquin slides.

\subsection{Matrix $W$}\label{sec:4.1}

Let $S$ be a skew tableau and $F_{i,j}$ denote the number of $0,1,\dots,j$'s in the top $i$ 
rows of $S$.
Put
\begin{equation}\label{eq:def_of_W}
\begin{aligned}
&W_{i,j}
:=F_{i,j}+F_{i,j+1}-F_{i-1,j}-F_{i+1,j+1}\\
&=\sharp \{0,1,\dots,j\mbox{'s in the } i^{\mathrm{th}} \mbox{ row}\}-\sharp\{0,1,\dots,(j+1)\mbox{'s in the } (i+1)^{\mathrm{th}} \mbox{ row}\}.
\end{aligned}
\end{equation}
By definition of skew tableaux, we find $W_{i,j}$ must be nonnegative and the sum $\sum_{p\geq 0}W_{i+p,j+p}$ satisfies
\[
\textstyle\sum_{p\geq 0}W_{i+p,j+p}=\sharp \{0,1,\dots,j\mbox{'s in the $i^{\mathrm{th}}$ row}\}.
\]

A skew tableau $S$ of shape $\lambda/\mu$ can be identified with the increasing sequence
\[
\mu=\lambda^{(0)}\subset \lambda^{(1)}\subset \lambda^{(2)}\subset\dots\subset \lambda^{(N)}=\lambda,
\]
where $\lambda^{(j)}$ is the sub-diagram of $S$ in which one of $0,1,2,\dots,j$ is filled\footnote{An empty box is regarded as a box with $0$.}.
Each skew diagram $\lambda^{(j+1)}/\lambda^{(j)}$ does not contain no two boxes in each column.
Obviously, we have
\[
\lambda_i^{(j)}=\textstyle\sum_{p\geq 0}W_{i+p,j+p},\qquad W_{i,j}=\lambda_i^{(j)}-\lambda_{i+1}^{(j+1)}.
\]

The $W_{i,j}$ satisfies the following conditions:
\begin{gather}
\mbox{
There exists some $J$ such that \ 
$j>J\Rightarrow W_{i,j}=W_{i,j+1}$ for all $i$}.\label{eq:cond1}\\
\mbox{There exists some $I$ such that\
 $i>I\Rightarrow W_{i,j}=0$ for all $j$}.\label{eq:cond1.5}\\
\textstyle \sum_{p\geq 0}W_{i+p,j+p}\geq \sum_{p\geq 0}W_{i+1+p,j+p}
\qquad (\Leftrightarrow \lambda_i^{(j)}\geq \lambda_{i+1}^{(j)}).\label{eq:cond2}
\end{gather}
Set
\begin{equation*}
\begin{gathered}
\Omega:= (\mbox{the set of skew tableaux}),\\
\mathfrak{X}:=\{(W_{i,j})_{
\kasane{i\geq 1}{j\geq 0 \hfill}
}\,\vert\,W_{i,j}\in \ZZ_{\geq 0}\ \mbox{that satisfies (\ref{eq:cond1}), (\ref{eq:cond1.5}), (\ref{eq:cond2})}
\}.
\end{gathered}
\end{equation*}
Then there exists the map $W:\Omega\to \mathfrak{X}$ that assigns a skew tableau $S$ with the matrix $(W_{i,j})$ defined by (\ref{eq:def_of_W}).
\begin{prop}\label{prop:bijection}
$W$ is bijective.
\end{prop}
\begin{proof}
By the correspondence $\lambda_i^{(j)}\mapsto \Delta_{i,j}$, $\Omega$ can be seen as a subset of
\[
\tilde{\Omega}:= \{(\Delta_{i,j})_{\kasane{i\geq 1}{j\geq 0}}\,\vert\, 
\Delta_{i,j}\in \ZZ_{\geq 0},\ \Delta_{i,j}\geq \Delta_{i+1,j+1}\geq \Delta_{i+2,j+2}\geq \dots \to 0,\ (\forall i,j)\}.
\]
We also regard $\mathfrak{X}$ as a subset of
\[
\tilde{\mathfrak{X}}:=\{(W_{i,j})_{
\kasane{i\geq 1}{j\geq 0 \hfill}
}\,\vert\,W_{i,j}\in \ZZ_{\geq 0},\ \mbox{the sum }\textstyle\sum_{p\geq 0}W_{i+p,j+p} \mbox{ converges for each $i,j$} \}.
\]
Then the map
\[
\tilde{\Omega}\to \tilde{\mathfrak{X}};\qquad (\Delta_{i,j})_{i,j}\mapsto (\Delta_{i,j}-\Delta_{i+1,j+1})_{i,j}
\]
is obviously bijective
(the inverse is $(W_{i,j})_{i,j}\mapsto (\sum_{p\geq 0} W_{i+p,j+p})_{i,j}$).
Its restriction to $\Omega$ coincides with $W$.
Because $\Omega$ contains the inverse image of $\mathfrak{X}$, $W$ must be bijective.
\end{proof}

\subsection{Jeu de taquin $\varphi_k$ }\label{sec:4.2}

From the statement in the previous paragraph, we always identify $\Omega\leftrightarrow \mathfrak{X}$.
We will construct a map $\varphi_k:\mathfrak{X}\to \mathfrak{X}$ for any positive integer $k$ that is a tropical counterpart of the jeu de taquin starting from $k^\mathrm{th}$ row.

For any $W=(W_{i,j})\in \mathfrak{X}$, we define $W^+:=\varphi_k(W)$, the image of $W$ by $\varphi_k$, by the following manner:
\begin{enumerate}
\def\labelenumi{(\theenumi).}
\item Set $\bm{Q}_0=(Q_{1,0},Q_{2,0},\dots):=(0,\dots,0,
\stackrel{\stackrel{k}{\vee}}{1},
0,\dots)$.
\item\label{mini:3} 
When the vector $\bm{Q}_j=(Q_{1,j},Q_{2,j},\dots)$ is already defined for $j\in \ZZ_{\geq 0}$, define $\bm{Q}_{j+1}=(Q_{1,j+1},Q_{2,j+1},\dots)$ and $\bm{W}^+_j=(W^+_{1,j},W^+_{2,j},\dots)$ by the formula
\begin{equation}\label{eq:R-matrix}
\begin{cases}
Q_{i+1,j+1}:=(\min[Q_{i+1,j},W_{i+1,j}]-\min[Q_{i,j},W_{i,j}])+Q_{i,j},\\
W_{i,j}^+:=(\min[Q_{i+1,j},W_{i+1,j}]-\min[Q_{i,j},W_{i,j}])+W_{i,j},
\end{cases}
\end{equation}
where $Q_{0,j}=0$, $W_{0,j}=+\infty$.
(Compare with (\ref{eq:evolution})).
\item Repeat (2) to obtain $Q_{i,j}$ and $W^+_{i,j}$ for all $i,j$.
\end{enumerate}

Equation (\ref{eq:R-matrix}) can be seen as a kind of recurrence formula, the inputs of which are $\bm{Q}_j$ and $\bm{W}_j$, and the outputs are $\bm{Q}_{j+1}$ and $\bm{W}^+_j$.
To understand this situation, it is convenient to draw the diagram 
\[
\sayou{ & \bm{W}_j & \\
\bm{Q}_j & \+ & \bm{Q}_{j+1}\\
 & \bm{W}^+_j
},
\]
where the inputs are written on the northwest side and the outputs are written on the southeast side.
Then the procedure that is presented above can be symbolically displayed as
\begin{equation}\label{eq:diagram-yoko}
\sayou{
 & \bm{W}_0 & & \bm{W}_1 & & \bm{W}_2 & & \bm{W}_3 & \\
\bm{Q}_0 & \+ & \bm{Q}_1 & \+ & \bm{Q}_2 & \+ & \bm{Q}_3 & \+ &\cdots\\
 & \bm{W}^+_0 & & \bm{W}^+_1 & & \bm{W}^+_2 & & \bm{W}^+_3 & \\
}.
\end{equation}

The map $\varphi_k$ also admits a diagrammatic interpretation as follows:
\begin{itemize}
\item Write a matrix $W=(W_{i,j})$ down as Fig \ref{fig:example1}.
\item Draw a path on the matrix by the following rule: 
\begin{itemize}
\item The path starts from the $(k,0)^{\mathrm{th}}$ position.
\item When the path reaches at the $(i,j)^{\mathrm{th}}$ position, it extends to the lower right neighbor if $W_{i,j}=0$, or to the right neighbor if $W_{i,j}\neq 0$.
\end{itemize}
\item Decrease all non-zero numbers on the path by one, and increase all the numbers at the upper neighbor of the decreased numbers by one.
The matrix given by this procedure coincides with $\varphi_k(W)$.
\item The matrix $Q=(Q_{i,j})_{i,j}$ is given by putting $Q_{i,j}=1$ if the path goes through the $(i,j)^\mathrm{th}$ position, and $Q_{i,j}=0$ otherwise.
\end{itemize}

\begin{figure}[htbp]
\centering
$W^0=
\left(
\haiti{
1\r & 1\r & 0 \d& 0 & 0 & 0 \\
1 & 1 & 2 & 2 \r& 1\r & 1 \\
0 & 0 & 0 & 0 & 1 & 2 \\
0 & 0 & 0 & 0 & 0 & 0 
}
\right)
$,
$W^1=
\left(
\haiti{
0 & 0 & 0 & 1 & 1 & 1 \\
1 \r& 1 \r& 2 \r& 1 \r& 0 \d& 0 \\
0 & 0 & 0 & 0 & 1 & 2 \\
0 & 0 & 0 & 0 & 0 & 0 
}
\right)
$,\\
$W^2=
\left(
\haiti{
1 \r& 1 \r& 1 \r& 2 \r& 1 \r& 1 \\
0 & 0 & 1 & 0 & 0 & 1 \\
0 & 0 & 0 & 0 & 1 & 1 \\
0 & 0 & 0 & 0 & 0 & 0 
}
\right)
$,

$Q^0=
\left(
\haiti{
1 & 1 & 1 & 0 & 0 & 0 \\
0 & 0 & 0 & 1 & 1 & 1 \\
0 & 0 & 0 & 0 & 0 & 0 \\
0 & 0 & 0 & 0 & 0 & 0 
}
\right)
$,
$Q^1=
\left(
\haiti{
0 & 0 & 0 & 0 & 0 & 0 \\
1 & 1 & 1 & 1 & 1 & 0 \\
0 & 0 & 0 & 0 & 0 & 1 \\
0 & 0 & 0 & 0 & 0 & 0 
}
\right)
$,\\
$Q^2=
\left(
\haiti{
1 & 1 & 1 & 1 & 1 & 1 \\
0 & 0 & 0 & 0 & 0 & 0 \\
0 & 0 & 0 & 0 & 0 & 0 \\
0 & 0 & 0 & 0 & 0 & 0 
}
\right)
$

\caption{Time evolution rule of $W^t$ and $Q^t$. 
$W^1=\varphi_1(W^0)$, $W^2=\varphi_2(W^1)$.
The path at time $t=2$ corresponds with $\varphi_1$.
}
\label{fig:example1}
\end{figure}
The $\varphi_k$ coincides with the jeu de taquin slide that starts from the $k^\mathrm{th}$ row. 
In fact, the data $W=(W_{i,j})$, $W^+=(W_{i,j}^+)$, $Q=(Q_{i,j})$ in (\ref{eq:diagram-yoko}) satisfy the relation (\ref{eq:evolution}) under the substitution $W_{i,j}= W_{i,j}^t$, $W_{i,j}^+=W_{i,j}^{t+1}$, $Q_{i,j}= Q_{i,j}^t$.
(Compare (\ref{eq:evolution}) with (\ref{eq:R-matrix}).)
Since 
\[
Q_{i,j}^t=F_{i,j}^t+F_{i-1,j}^{t+1}-F_{i-1,j}^t-F_{i,j}^{t+1}
=(F_{i,j}^t-F_{i-1,j}^t)-(F_{i,j}^{t+1}-F_{i-1,j}^{t+1})
\] 
(see (\ref{eq:defofQW})), the number of $j$'s in the $i^\mathrm{th}$ row decreases by $Q_{i,j}^t$ under the time evolution $t\mapsto t+1$.
This means that substituting
$\bm{Q}_0=(0,\dots,0,
\stackrel{\stackrel{k}{\vee}}{1},
0,\dots)$ is equivalent to removing an empty box from the $k^\mathrm{th}$ row, and is also equivalent to starting the jeu de taquin slide from the $k^\mathrm{th}$ row.

The $Q_{i,j}$ also has a diagrammatic interpretation. 
Consider the jeu de taquin slide $\varphi_k$.
Let $B_j$ be the position of the hole (see \S \ref{sec:appB}) at when all the numbers equal to or less than $j$ have been moved. 
Then, $Q_{i,j}=1$ if $B_j$ is in the $i^\mathrm{th}$ row, and $Q_{i,j}=0$ otherwise.

\subsection{Example}\label{sec:4.3}

The jeu de taquin slide
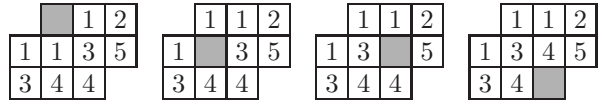
\[
\Tableau{\bl & \gray & 1 & 2 \\ 1 & 1 & 3 & 5 \\ 3 & 4 & 4 \\ }
\quad
\Tableau{\bl & 1 & 1 & 2 \\ 1 & \gray & 3 & 5 \\ 3 & 4 & 4 \\ }
\quad
\Tableau{\bl & 1 & 1 & 2 \\ 1 & 3 & \gray & 5 \\ 3 & 4 & 4 \\ }
\quad
\Tableau{\bl & 1 & 1 & 2 \\ 1 & 3 & 4 & 5 \\ 3 & 4 & \gray \\ }
\]
corresponds with the matrices
\[
W=
\left(
\haiti{
0 \d& 1 & 1 & 1 & 0 & 0 \\
0 & 2 \r& 1 \r& 0 \d& 0 & 1 \\
0 & 0 & 0 & 1 & 3 \r& 3
}
\right),\quad
Q=
\left(
\haiti{
1 & 0 & 0 & 0 & 0 & 0 \\
0 & 1 & 1 & 1 & 0 & 0 \\
0 & 0 & 0 & 0 & 1 & 1
}
\right).
\]
The matrix $W^+$ is given by
\[
W^+=
\left(
\haiti{
0 & 2 & 2 & 1 & 0 & 0 \\
0 & 1 & 0 & 0 & 1 & 2 \\
0 & 0 & 0 & 1 & 2 & 2
}
\right).
\]

\section{Tropical interpretation of rectification}\label{sec:5}

In the following two sections, we give an alternative proof of the uniqueness of a rectification~\cite[\S 1--\S 3]{fulton_1996}.
First, we introduce the definition of the rectification and its tropical interpretation (\S \ref{sec:5.1}).
Then, we construct a new useful combinatorial map that is equivalent to the rectification.
The construction is based on the theory of Noumi-Yamada's geometric tableaux (\S \ref{sec:5.2}).
Finally, we show a diagrammatic realization (\S \ref{sec:5.4}) and a commutation relation of the map (\S \ref{sec:5.5}).
We will see that the commutative diagram in \S \ref{sec:5.5} plays an important role in proving the uniqueness of a rectification. 

\subsection{Rectification}\label{sec:5.1}

Any skew tableau of shape $\lambda/\mu$ is led to a (non-skew) tableau by applying a finite sequence of jeu de taquin slides.
Repeating jeu de taquin slides is nothing but choosing inside corners repeatedly.
By putting numbers in such inside corners in decreasing order, one obtains a standard tableau of shape $\mu$.
For example, if we apply the sequence of jeu de taquin slides to the tableau 
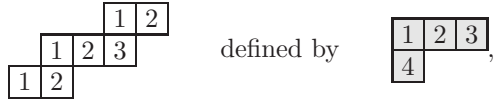
\[
\Tableau{
\bl & \bl & \bl & 1 & 2\\
\bl & 1 & 2 & 3\\
1 & 2
}\qquad \mbox{defined by}\qquad
\Tableau{
\lgray 1 & \lgray 2 & \lgray 3\\
\lgray 4
},
\]
we obtain the sequence
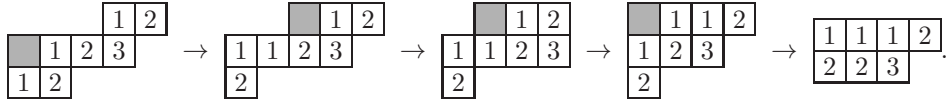
\[
\Tableau{
\bl & \bl & \bl & 1 & 2\\
\gray & 1 & 2 & 3\\
1& 2
}\ \to\ 
\Tableau{
\bl & \bl & \gray & 1 & 2\\
1 & 1 & 2 & 3\\
2
}\ \to\ 
\Tableau{
\bl &  \gray & 1 & 2\\
1 & 1 & 2 & 3\\
2
}\ \to\ 
\Tableau{
\gray &  1 & 1 & 2\\
1 & 2 & 3 \\
2
}\ \to\ 
\Tableau{
1 & 1 & 1 & 2\\
2 & 2 & 3  
}.
\]
While, another standard tableau
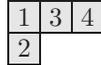
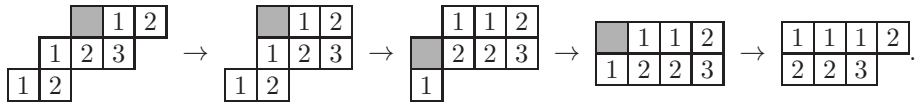
\[
\Tableau{
\lgray 1 & \lgray 3& \lgray 4\\
\lgray 2
}
\]
gives the sequence
\[
\Tableau{
\bl & \bl & \gray & 1 & 2\\
\bl & 1 & 2 & 3\\
1& 2
}\ \to\ 
\Tableau{
\bl & \gray & 1 & 2\\
\bl & 1 & 2 & 3\\
1& 2
}\ \to\ 
\Tableau{
\bl & 1 & 1 & 2\\
\gray & 2 & 2 & 3\\
1
}\ \to\ 
\Tableau{
\gray & 1 & 1 & 2\\
1 & 2 & 2 & 3
}\ \to\ 
\Tableau{
1 & 1 & 1 & 2\\
2 & 2 & 3
}.
\]
It is not a coincidence that the two tableaux at the rightmost are same.
In fact, it is known that any choice of standard tableau leads a unique tableau \cite[\S 1, Claim 2]{fulton_1996}.
\begin{defi}
Let $S$ be a skew tableau.
A rectification of $S$ is a tableau that is obtained by applying a finite sequence of jeu de taquin slides to $S$.
\end{defi}

With diagrammatic expressions as in \S \ref{sec:4.2}, the rectification can be displayed as
\begin{equation}\label{eq:diagram-of-rectification}
\sayou{
 & \bm{W}_0& &\bm{W}_1& &\bm{W}_2& \cdots \\
\bm{Q}_0^{(1)} & \+ & \bm{Q}_1^{(1)}& \+ & \bm{Q}_2^{(1)} & \+ &\cdots\\
 & \bm{W}^{(1)}_0& &\bm{W}^{(1)}_1& &\bm{W}^{(1)}_2& \cdots \\
 \bm{Q}_0^{(2)} & \+ & \bm{Q}_1^{(2)}& \+ & \bm{Q}_2^{(2)} & \+ &\cdots\\
  & \bm{W}^{(2)}_0& &\bm{W}^{(2)}_1& &\bm{W}^{(2)}_2& \cdots \\
  & \vdots& &\vdots& &\vdots& \\
\bm{Q}_0^{(\ell)} & \+ & \bm{Q}_1^{(\ell)}& \+ & \bm{Q}_2^{(\ell)} & \+ &\cdots\\
 & \bm{W}^{(\ell)}_0& &\bm{W}^{(\ell)}_1& &\bm{W}^{(\ell)}_2& \cdots \\
}.
\end{equation}
Each vector $\bm{Q}_0^{(\ell-i)}$ corresponds with the jeu de taquin slide starting at $\Tableau{\lgray i}$．
The vectors $\bm{W}^{(\ell)}_0,\bm{W}^{(\ell)}_1,\bm{W}^{(\ell)}_2,\dots$ at the bottom row correspond with the rectification.

\subsection{Noumi-Yamada's geometric tableau}\label{sec:5.2}

In \cite{noumi2004}, Noumi and Yamada introduced an interesting characterization of the {\it row bumping}~\cite{fulton_1996} in terms of tropical integrable systems.

For real vectors $I=(I_1,I_2,\dots)$ and $V=(V_1,V_2,\dots)$, we define the matrices $E(I)$, $F(V)$ of infinite size as 
\[
E(I)=\left(
\begin{array}{cccc}
I_1 & 1 &  &  \\ 
 & I_2 & 1 &  \\ 
 &  & I_3 & \ddots \\ 
 &  &  & \ddots
\end{array} 
\right),\quad
F(V)=\left(
\begin{array}{cccc}
1 &  &  &  \\ 
-V_1 & 1 &  &  \\ 
 & -V_2 & 1 &  \\ 
 &  & \ddots & \ddots
\end{array} 
\right).
\]
Moreover, for a vector $I'=(1,\dots,1,I_{k},I_{k+1},\dots)$ whose first ($k-1$) entries are $1$, we put
\[
E_k(I')=\left(
\begin{array}{cc}
\mathrm{Id}_{k-1} &  \\ 
 & E(I'')
\end{array} 
\right),\quad I''=(I_k,I_{k+1},\dots).
\]

Let us consider the equation
\begin{equation}\label{eq:mat-equation}
E(I_\ell)\cdots E(I_2) E(I_1)=E_1(J_1)E_2(J_2)\dots E_\ell(J_\ell),
\end{equation}
where $I_k=(I_{1,k},I_{2,k},\dots)$ and $J_k=(1,\dots,1,J_{k,k},J_{k+1,k},\dots)$ are real vectors for $k=1,2,\dots,\ell$.
It is proved by using Gaussian elimination method that (\ref{eq:mat-equation}) determines some rational map that corresponds $\{I_k\}_k$ to $\{J_k\}_k$.

The following theorem is given by Noumi-Yamada~\cite[\S 2]{noumi2004}:
\begin{thm}[Geometric tableau~\cite{noumi2004}]\label{thm:noumi-yamada}
Equation $(\ref{eq:mat-equation})$ possesses the following properties:
\begin{enumerate}
\def\labelenumi{$(\roman{enumi})$}
\item 
The correspondence $\{I_{i,j}\}\mapsto \{J_{i,j}\}$ is a subtraction-free rational map, that is, every $J_{i,j}$ is expressed as a subtraction-free rational function of $\{I_{i,j}\}$.
This implies the existence of the tropicalization of the map. 
\item Let $Q_{i,j}=\o{I_{i,j}}$ and $P_{i,j}=\o{J_{i,j}}$ be tropical variables.
Then the ``tropicalized'' map $\{Q_{i,j}\}\mapsto \{P_{i,j}\}$ has the following combinatorial interpretation:
Let $\bm{Q}_j=(Q_{1,j},Q_{2,j},\dots)$ be the vector whose $\alpha_j^{\mathrm{th}}$ entry is $1$ and the others are $0$.
Then $P_{i,j}$ 
equals to the number of $j$'s in the $i^\mathrm{th}$ row of the tableau
\[
\Tableau{\alpha_1}\la \Tableau{\alpha_2}\la\cdots \la\Tableau{\alpha_\ell}.
\]
\end{enumerate}
\end{thm}


\subsection{$P$-tableau associated with standard tableau}\label{sec:5.3}

We now proceed to the discrete Toda equation (\ref{eq:Laxform}), which is equivalent to
\begin{equation}\label{eq:FE=EF}
F(V_j^{t+1})E(I_j^t)=E(I_{j+1}^t)F(V_j^{t}),
\end{equation}
where $I_j^t=(I_{1,j}^t,I_{2,j}^t,\dots)$ and $V_j^t=(V_{1,j}^t,V_{2,j}^t,\dots)$. 
With regarding $I_j^t,V_j^t$ as inputs and $I_{j+1}^t,V_j^{t+1}$ as outputs, we display (\ref{eq:FE=EF}) diagrammatically as
$
\sayou{
& V_j^t& \\
I_j^t&\- & I_{j+1}^{t}\\
& V_{j}^{t+1}
}$.
The diagram
$\sayou{
& \bm{W}_j^t& \\
\bm{Q}_j^t&\+ & \bm{Q}_j^{t+1}\\
& \bm{W}_{j+1}^t 
}$
in \S\ref{sec:4.2} is nothing but its tropicalization.

With this idea, we associate the equation
\begin{equation}\label{eq:FEE=EEF}
F(V')E(I^{(k)})\cdots E(I^{(1)})=E({I^{(k)}}')\cdots E({I^{(1)}}')F(V)
\end{equation}
with the vertical diagram
\begin{equation}\label{eq:vertical-diagram}
\sayou
{ & V & \\
I^{(1)} & \- & {I^{(1)}}'\\
& V^{(1)} & \\
I^{(2)} & \- & {I^{(2)}}'\\
& V^{(2)} & \\
& \vdots & \\
I^{(k)} & \- & {I^{(k)}}'\\
& V'\hbox to 0pt{$=V^{(k)}$} & 
}.
\end{equation}
Let us transform the upper triangle matrices on the both sides of (\ref{eq:FEE=EEF}) into the geometric tableaux (\ref{eq:mat-equation}) form:
\[
E(I^{(k)})\cdots E(I^{(1)})=E_1(J_1)\cdots E_k(J_k),\quad 
E({I^{(k)}}')\cdots E({I^{(1)}}')=E_1(J'_1)\cdots E_k(J'_k).
\]
This leads the new equation
\begin{equation}\label{eq:matrix-rectification}
F(V')E_1(J_1)\cdots E_k(J_k)=E_1(J_1')\cdots E_k(J_k')F(V),
\end{equation}
or equivalently,
\begin{equation}\label{eq:matrix-rectification'}
E_1(J_1)\cdots E_k(J_k)F(V)^{-1}=F(V')^{-1}E_1(J_1')\cdots E_k(J_k')
\tag{\ref{eq:matrix-rectification}$'$}
\end{equation}
which we will display diagrammatically as
\[
\sayou{&V&\\
(J_1,\dots,J_k)&\-&(J'_1,\dots,J'_k)\\
&V'&
}.
\]

\begin{prop}\label{prop:subtraction-free}
Equation $(\ref{eq:matrix-rectification})$ determines a birational map $\{J_i,V\}\leftrightarrow \{J_i',V'\}$, which has the following proprieties:
\begin{enumerate}
\def\labelenumi{$(\theenumi)$}
\item Each entry of $J'_i$ and $V'$ is expressed as a subtraction-free rational function of entries of $J_i$ and $V$.
\item Each entry of $J_i$ and $V$ is expressed as a subtraction-free rational function of entries of $J'_i$ and $V'$.
\end{enumerate}
\end{prop}
\begin{proof}
By straightforward calculations, one verifies that the equation \[
E_i(A)F(B)^{-1}=F(B')^{-1}E_i(A')
\] 
determines a birational map $\{A,B\}\leftrightarrow \{A',B'\}$ for any $i$.
Here each entry of $A',B'$ is expressed as a subtraction-free rational function of entries of $A$ and $B$.
With this fact, we have
\begin{align*}
E_1(J_1)\cdots E_k(J_k)F(V)^{-1}&=E_1(J_1)\cdots E_{k-1}(J_{k-1})F(V^{(1)})^{-1}E_k(J'_k)\\
&=E_1(J_1)\cdots E_{k-2}(J_{k-2})F(V^{(2)})^{-1}E_{k-1}(J'_{k-1})E_k(J'_k)\\
&=\cdots=F(V^{(k)})^{-1}E_1(J'_1)\cdots E_{k-1}(J'_{k-1})E_k(J'_k),
\end{align*}
where each entry of $J_i'$ and $V^{(i)}$ are subtraction-free function of positive vectors of $J_i$ and $V$.
These statements are also true if $\{J_i,V\}$ and $\{J_i',V'\}$ are exchanged with each other.
\end{proof}

Since all the rational maps in Proposition \ref{prop:subtraction-free} are subtraction-free, the correspondence $\{J_i,V\} \leftrightarrow \{J_i',V'\}$ is one-to-one if we restrict ourselves to real and positive $J_i,V$.
Let $P_i=\o{J_i}$ and $W_i=\o{V}$ be tropical variables.
By the principle of tropicalization (Proposition \ref{prop:axiom-trop-principle}), we obtain the one-to-one tropical map $\{P_i,W\}\leftrightarrow \{P_i',W'\}$, which will be diagrammatically written as
\begin{equation}\label{eq:diagram_WP=PW}
\sayou{&W&\\
(P_1,\dots,P_k)&\+&(P'_1,\dots,P_k')\\
&W'&
}.
\end{equation}

Applying Theorem \ref{thm:noumi-yamada} (ii) to the data $(P_1,\dots,P_k)$, we can identify it with some Young tableau, which we will call the {\it $P$-tableau}.
For example, let us consider the tableau at the beginning of \S\ref{sec:5.1}:
\[
\Tableau{
\bl & \bl & \bl & 1 & 2\\
\bl & 1 & 2 & 3\\
1 & 2
},\qquad
W=\left(
\haiti{
1&1&1&1\\
0&0&1&2\\
0&1&2&2
}
\right)
\]
and the sequence of jeu de taquin slides defined by the standard tableau
\[
\Tableau{
\lgray 1 & \lgray 2 & \lgray 3\\
\lgray 4
}.
\]
The order of row numbers from which jeu de taquin slides start is $(2^\mathrm{nd},1^\mathrm{st},1^\mathrm{st},1^\mathrm{st})$.
The $P$-tableau associated with this sequence is 
\[
\Tableau{2}\la \Tableau{1}\la \Tableau{1}\la \Tableau{1}=\Tableau{1&1&1\\2}.
\]
When the jeu de taquin slides are applied, the outside corners in the $3^\mathrm{rd},1^\mathrm{st},2^\mathrm{nd},3^\mathrm{rd}$ rows are removed. 
Finally we obtain the rectified tableau
\[
\Tableau{1&1&1&2\\2&2&3},\qquad 
W=\left(
\haiti{
0&1&1&1\\
0&0&2&3\\
0&0&0&0
}
\right).
\]
With use of $P$-tableaux, the diagram (\ref{eq:diagram-of-rectification}) is now rewritten as
\begin{equation}\label{eq:example-of-action}
\sayou{
&(1,0,0)&&(1,0,1)&&(1,1,2)&&(1,2,2)&\\
\Tableau{1&1&1\\2}&\+&\Tableau{1&2&2\\3}&\+&\Tableau{1&2&3\\3}&\+&\Tableau{1&2&3\\3}&\+&\Tableau{1&2&3\\3}\\
&(0,0,0)&&(1,0,0)&&(1,2,0)&&(1,3,0)&
}.
\end{equation}
The tableau on the rightmost position corresponds to the sequence of row insertions
\[
\Tableau{3}\la \Tableau{1}\la \Tableau{2}\la \Tableau{3}=\Tableau{1&2&3\\3}.
\]

Note that the $P$-tableau does not change if one replaces the standard tableau with
\begin{equation}\label{eq:replace_standard_tableaux}
\Tableau{\lgray 1&\lgray 2&\lgray 4\\\lgray 3}\
(=1^\mathrm{st},2^\mathrm{nd},1^\mathrm{st},1^\mathrm{st})
\qquad \mbox{or}\qquad 
\Tableau{\lgray 1&\lgray 3&\lgray 4\\\lgray 2}\ 
(=1^\mathrm{st},1^\mathrm{st},2^\mathrm{nd},1^\mathrm{st})
.
\end{equation}
This is the essential reason why the rectification is unique.

\subsection{Diagrammatic expression of the map (\ref{eq:diagram_WP=PW})}\label{sec:5.4}

As we have seen above, the matrix equation (\ref{eq:matrix-rectification'}) 
boils down to the simple equation
\begin{equation}\label{eq:EF=FE}
E_p(J)F(V)^{-1}=F(V')^{-1}E_p(J'),
\quad
V=(V_1,V_2,\dots),\ 
J=(1,\dots,1,J_p,J_{p+1},\dots)
\end{equation}
for $1\leq p\leq k$ (see the proof of Proposition \ref{prop:subtraction-free}).
Let $\{J,V\}\mapsto \{J',V'\}$ be the subtraction-free rational map defined by (\ref{eq:EF=FE}).
Then, its tropicalization $\{P,W\} \mapsto \{P',W'\}$ can be obtained by almost same calculations that we have done in \S \ref{sec:2}.
In fact, we have
\[
\begin{aligned}
&P'_{i+1}=
\begin{cases}
\min[P_{i+1},W_{i+1}], & i=p-1\\
(\min[P_{i+1},W_{i+1}]-\min[P_{i},W_{i}])+P_{i}, & i\geq p
\end{cases},\\
&W'_{i}=
\begin{cases}
W_i, & i<p-1\\
\min[P_{i+1},W_{i+1}]+W_i, & i=p-1 \\
(\min[P_{i+1},W_{i+1}]-\min[P_{i},W_{i}])+W_{i}, & i\geq p
\end{cases}.
\end{aligned}
\]
Note that, by putting $P_i\equiv 0$ for all $1\leq i<p$, they are simplified as
\begin{equation}\label{eq:rel_trop}
\begin{aligned}
&P'_{i+1}=(\min[P_{i+1},W_{i+1}]-\min[P_{i},W_{i}])+P_{i},\\
&W'_{i}=(\min[P_{i+1},W_{i+1}]-\min[P_{i},W_{i}])+W_{i}.
\end{aligned}
\end{equation}
(Compare with (\ref{eq:evolution}).)

The system (\ref{eq:rel_trop}) can be realized by the \textit{kicker-and-ball model}~\cite{iwao2018discrete}, which was first introduced as a tropicalization (= ultradiscretization) of the \textit{discrete relativistic Toda equation}.
Let $P=(0,\dots,0,P_p,P_{p+1},\dots)$ and $W=(W_1,W_2,\dots)$ be sequences of non-negative integers.
Consider infinitely many boxes aligned in a half-line towards the right and put the following objects on them:
\begin{itemize}
\item $P_i$ kickers at the $i^\mathrm{th}$ site from the left.
\item $W_i$ balls at the $i^\mathrm{th}$ site from the left.
\end{itemize}
For example, if $P=(0,2,0,2,1,0,0,\dots)$ and $W=(0,3,1,1,0,0,0,\dots)$, we draw
\[
\ktableau{ & \b\!\k\!\b\!\k\!\b & \b & \k\!\b\!\k &\k & & }\cdots
\]

To obtain $\{P',W'\}$, we move the kickers ($k$) and the balls ($\circ$) by the following rules:
\begin{itemize}
\item Kickers who stand nearby a ball kick one out into the box on their left.
(A ball that is kicked out from the leftmost box will disappear.)
\item Kickers who have no balls to kick out proceed to the box on their right.
\end{itemize}
For the example above, we obtain
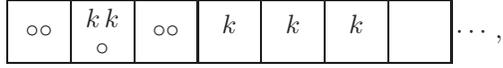
\[
\ktableau{\b\b & \k\!\b\!\k & \b\b & \k &\k & \k& }\cdots,
\]
and then we find $P'=(0,2,0,1,1,1,0,\dots)$ and $W'=(2,1,2,0,0,0,0,\dots)$.
This situation is expressed as
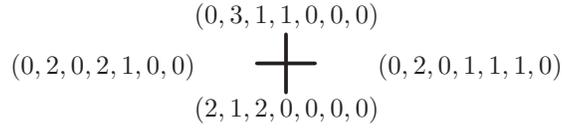
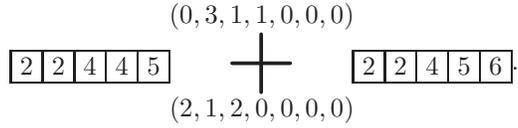
\[
\sayou{
&(0,3,1,1,0,0,0)&& \\
(0,2,0,2,1,0,0)&\+&(0,2,0,1,1,1,0) \\
&(2,1,2,0,0,0,0)&
}
\]
or equivalently,
\[
\sayou{
&(0,3,1,1,0,0,0)&& \\
\Tableau{2&2&4&4&5}&\+&\Tableau{2&2&4&5&6} \\
&(2,1,2,0,0,0,0)&
}.
\]

According to (\ref{eq:matrix-rectification'}), the whole procedure to obtain $\{P'_i,W'\} $ from $\{P_i,W\}$ is written diagrammatically as 
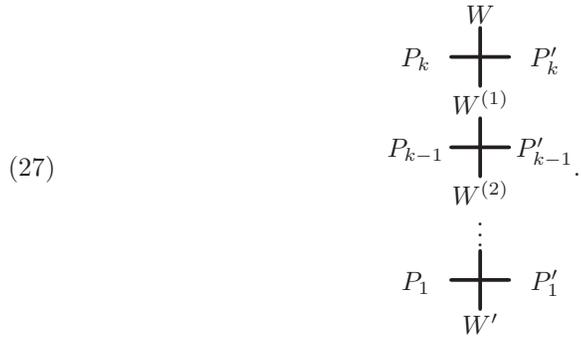
\begin{equation}\label{eq:diagram_new_map}
\sayou{&W&\\
P_k&\+&P_k'\\
&W^{(1)}&\\
P_{k-1}&\+&P_{k-1}'\\
&W^{(2)}&\\
& \vdots &\\
P_1&\+&P_1'\\
&W'&
}.
\end{equation}
Here each step is a map described in the previous paragraph.
For example, the diagram
$
\sayou{
&(1,0,0)& \\
\Tableau{1&1&1\\2}&\+&\Tableau{1&2&2\\3}\\
&(0,0,0)&
}
$
is decomposed as
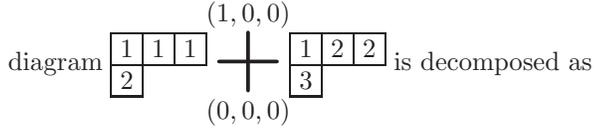
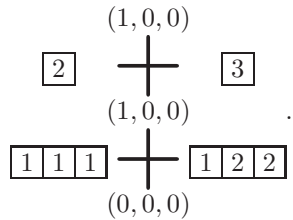
\[
\sayou{
&(1,0,0)& \\
\Tableau{2}&\+&\Tableau{3}\\
&(1,0,0)&\\
\Tableau{1&1&1}&\+&\Tableau{1&2&2}\\
&(0,0,0)&
}.
\]

\subsection{Summary: a commutative diagram}\label{sec:5.5}

Our results are summed up by the commutative diagram:
\begin{equation}\label{eq:diagram_commute}
\def\myspace{2.5em}%
\xymatrix{
\hspace{\myspace}\{Q,W\}\hspace{\myspace}\ar@{|->}[r]^{(\ref{eq:diagram-of-rectification})} 
\ar@{|->}[d]_{\mathrm{Row\ insertion}} 
& \hspace{\myspace}\{Q',W'\}\hspace{\myspace}\ar@{|->}[d]^{\mathrm{Row\ insertion}}\\
\hspace{\myspace}\{P,W\}\hspace{\myspace}\ar@{|->}[r]^{(\ref{eq:diagram_WP=PW})\mathrm{\,(equivalently,\, }(\ref{eq:diagram_new_map}))} & \hspace{\myspace}\{P',W'\}\hspace{\myspace}
}.
\end{equation}
Here the matrix $W$ is equivalent to a skew tableau (Proposition \ref{prop:bijection}), and $Q$ represents a sequence of row numbers at where jeu de taquin slides start (\S \ref{sec:5.1}).

For example, we have
\[
\xymatrix{
\left\{
(2^\mathrm{nd},1^\mathrm{st},1^\mathrm{st},1^\mathrm{st}),
\Tableau{
\bl & \bl & \bl & 1 & 2\\
\bl & 1 & 2 & 3\\
1 & 2
}
\right\}
\ar@{|->}[r]
\ar@{|->}[d]
&
\left\{
(3^\mathrm{rd},1^\mathrm{st},2^\mathrm{nd},3^\mathrm{rd}),
\Tableau{1&1&1&2\\2&2&3}
\right\}
\ar@{|->}[d]
\\
\left\{
\Tableau{1&1&1\\2},
\Tableau{
\bl & \bl & \bl & 1 & 2\\
\bl & 1 & 2 & 3\\
1 & 2
}
\right\}
\ar@{|->}[r]
&
\left\{
\Tableau{1&2&3\\3},
\Tableau{1&1&1&2\\2&2&3}
\right\}
}.
\]
Note that if one replaces $(2^\mathrm{nd},1^\mathrm{st},1^\mathrm{st},1^\mathrm{st})$ at the top left of this diagram with $(1^\mathrm{st},2^\mathrm{nd},1^\mathrm{st},1^\mathrm{st})$ or $(1^\mathrm{st},1^\mathrm{st},2^\mathrm{nd},1^\mathrm{st})$, the bottom row does not change at all (see (\ref{eq:replace_standard_tableaux})).

\section{Application: Proof of the uniqueness of a rectification}\label{sec:6}

By using the results in the previous section, the uniqueness of a rectification now boils down to an relatively easy and purely combinatorial lemma, which we will show below.

A \textit{reverse lattice word} (or a \textit{Yamanouchi word}) is a sequence of positive integers $t_1,t_2,\dots,t_N$ that satisfies the following inequality for each $p$ and $i$:
\[
\sharp(\mbox{$i$'s contained in $t_p,t_{p+1},\dots,t_N$})\geq 
\sharp(\mbox{$(i+1)$'s contained in $t_p,t_{p+1},\dots,t_N$}).
\]
\begin{defi}
Let $U(\mu)$ denote the tableau of shape $\mu$ whose $i^\mathrm{th}$ row contains only $i$ for all $i$.
\end{defi}
\begin{lemma}[See Fulton \protect{\cite[\S 5.2, Lemma 1]{fulton_1996}}]\label{lemma:rev_lattice}
If $t_1,t_2,\dots,t_N$ is a reverse lattice word, we have
\[
\Tableau{t_1}\leftarrow
\Tableau{t_2}\leftarrow
\cdots\leftarrow
\Tableau{t_N}
=
U(\lambda)
\]
for some Young diagram $\lambda$.
\end{lemma}
\begin{proof}[Proof] \footnote{An excellent proof of Lemma \ref{lemma:rev_lattice} can be found in Fulton's book~\cite{fulton_1996} but we give another elementary proof in this article to avoid the possibility of circular reasoning.
See Remark \ref{rem:known-things}.} 
We first focus on the $1^\mathrm{st}$ row.
Row-inserting $t_1,\dots,t_N$, we obtain a $1^\mathrm{st}$ row $\Tableau{a_1&a_2&\cdots &a_r}$ and a series $s_1,\dots,s_q$ that are bumped from the $1^\mathrm{st}$ row ($\{a_1,\dots,a_r\}\sqcup\{s_1,\dots,s_q\}=\{t_1,\dots,t_N\}$).
We will prove $a_1=a_2=\dots=a_r=1$ and 
\[
\sharp(\mbox{$i$'s contained in $s_p,s_{p+1},\dots,s_q$})\geq 
\sharp(\mbox{$(i+1)$'s contained in $s_p,s_{p+1},\dots,s_q$})
\]
for all $i>1$ and $p$.
(In other words, $s_1-1,s_2-1,\dots,s_q-1$ is a reverse lattice word.)
Let $T_k:=\Tableau{t_1}\leftarrow
\Tableau{t_2}\leftarrow
\cdots\leftarrow
\Tableau{t_k}$ and
\[
\begin{gathered}
L^k_i:=
\sharp(\mbox{$i$'s contained in $t_{k+1},t_{k+2},\dots,t_N$}),\\
P^k_i:=\sharp(\mbox{$i$'s contained in the $1^\mathrm{st}$ row of $T_k$}).
\end{gathered}
\]
Since $t_1,\dots,t_N$ is a reverse lattice word, the sequence $L^k_1,L^k_2,\dots$ is weakly decreasing for each $k$.
Let $X^k_i:=L^k_i+P^k_i$.
By the definition, the row bumping algorithm is explicitly characterized by $L^{k+1}_i=L^{k}_i-\delta_{i,t_k}$ and $X^{k+1}_i=X^k_i-\delta_{i,\alpha}$, where $\alpha$ is the minimum number with ($\alpha>t_k$ and $P^k_\alpha> 0$). 
($\delta_{i,\alpha}\equiv 0$ if there exists no such $\alpha$.)
Obviously, $X^k_i\geq L^k_i$ and $X^{k}_i\geq X^{k+1}_i$ hold.
We will show $X^k_{i+1}\leq L^k_i$, which implies that $X^k_1,X^k_2,\dots$ is weakly decreasing.
When $k=0$, the claim is trivial because $P^0_i=0$ for all $i$.
Assume that the claim is true for some $k\geq 0$.
Therefore, we have
(i) $i\neq t_k\Rightarrow L^{k+1}_i-X^{k+1}_{i+1}= L^{k}_i-X^{k+1}_{i+1}\geq L^{k}_i-X^{k}_{i+1}$,
(ii) $(i= t_k\mbox{ and }P^k_{i+1}>0)\Rightarrow \alpha=i+1=t_k+1\Rightarrow
L^{k+1}_i-X^{k+1}_{i+1}=L^{k}_i-X^{k}_{i+1}
$,
and
(iii) $(i= t_k\mbox{ and }P^k_{i+1}=0)\Rightarrow P^{k+1}_{i+1}=0\Rightarrow
L^{k+1}_i-X^{k+1}_{i+1}=L^{k+1}_i-L^{k+1}_{i+1}
$.
In each case, we obtain $L^{k+1}_i-X^{k+1}_{i+1}\geq 0$ by the induction hypothesis.
Then $X_{i+1}^k\leq L^k_i$ is proved for all $k\geq 0$.

As $L^{N}_i=0$ for all $i$, we have $X_2^N=X_3^N=\cdots=0$.
On the other hand, we have $X_1^N=L^0_1$ because any $1$ cannot be bumped from the $1^\mathrm{st}$ row.
This means $s_1,\dots,s_q$ contains no $1$'s, and the row $\Tableau{a_1&a_2&\cdots &a_r}$ satisfies $a_1=\dots=a_r=1$.
Moreover, because $X^k_i$ is equal to the number of $i$'s contained in $s_{p+1},\dots,s_q$ ($p$ is the number of numbers bumped from the $1^\mathrm{st}$ row in the first $k$ steps) for $i\geq 2$, and $X^k_1,X^k_2,\dots$ is weakly decreasing, we find $s_1-1,s_2-1,\dots,s_q-1$ is a reverse lattice word.

Repeating this procedure, we show that the $j^\mathrm{th}$ row of $T_N$ contains only $j$ for all $j$.
We conclude $T_N=U(\lambda)$ with $\lambda_i=L^0_i$.
\end{proof}
\begin{cor}\label{cor:U(mu)}
The $P$-tableau associated with any standard tableau of shape $\mu$ must be $U(\mu)$.
\end{cor}
\begin{proof}
Let $S$ be a standard tableau of shape $\mu$ and size $N$.
Assume $(N-i+1)$ is contained in the $t_i^\mathrm{th}$ row of $S$.
The sequence $t_1,t_2,\dots,t_N$ should be a reverse lattice word because the subdiagram of $S$ that consists of $1,\dots,N-i$ is still a Young diagram for any $i$.
Therefore, the $P$-tableau associated with $S$ is $U(\lambda)$ for some $\lambda$.
Obviously, $\lambda=\mu$.
\end{proof}

The uniqueness of a rectification is now almost trivial from the diagram (\ref{eq:diagram_commute}) and Corollary \ref{cor:U(mu)}.

\begin{rem}\label{rem:known-things}
The notion of the ``$P$-tableau associated with a standard tableau'' is equivalent to the ``$P$-tableau of a reverse lattice word,'' which is well-known in the context of combinatoric.
For example, in Fulton's textbook~\cite[\S 5.3]{fulton_1996}, the standard tableau whose $i^\mathrm{th}$ row contains $t_i$ is denoted by $U(w)$ for a reverse lattice word $w=t_1t_2\dots t_N$.
Therein, the $P$-tableau associated with $U(w)$ is denoted by $P(w)$.
\end{rem}

\subsection*{Acknowledgment}
This work was supported by the Research Institute for Mathematical Sciences, an International Joint Usage/Research Center located in Kyoto University.
This work was also supported by JSPS KAKENHI:19K03605.

\appendix

\section{Notes on Mathematical Logic}\label{sec:appA}

In this appendix, we shortly review a few of basic notations of mathematical logic and we give a proof of Lemma \ref{prop:inherit}.

\subsection{Basic definitions}

In this section, we follow the notations in the textbooks of mathematical logic~\cite{marker2006model,tent2012course}.
\begin{defi}
A {\it language} $\mathcal{L}$ is a set of function symbols, relation symbols, and constant symbols.
Each function symbol $f$ is associated with a natural number $n_f$, and each relation symbol $R$ is associated with a natural number $n_R$.
\end{defi}
We say that ``$f$ is an $n_f$-ary function'' and ``$R$ is an $n_R$-ary relation.''

\begin{defi}
An $\mathcal{L}$-structure $\mathcal{M}$ is a collection of the following objects:
\begin{itemize}
\item An non-empty set $M$, which is called the domain.
\item A map $f^\mathcal{M}:M^{n_f}\to M$ for each function symbol $f\in \mathcal{L}$.
\item A set $R^\mathcal{M}\subset M^{n_R}$ for each relation symbol $R\in \mathcal{L}$.
\item An element $c^\mathcal{M}\in M$ for each constant symbol $c\in \mathcal{L}$.
\end{itemize}
These $f^\mathcal{M}$, $R^\mathcal{M}$, $c^\mathcal{M}$ are called an {\it interpretation} of $f,R,c$, respectively.
\end{defi}
We often write ``$R^\mathcal{M}(m_1,\dots,m_n)$'' instead of ``$(m_1,\dots,m_n)\in R^{\mathcal{M}}$.''

\begin{defi}
Let $\mathcal{M},\mathcal{N}$ be $\mathcal{L}$-structures and $M,N$ be their domains.
A map $h:M\to N$ is called a {\it morphism of $\mathcal{L}$-structures} if:
\begin{itemize}
\item $h(f^\mathcal{M}(m_1,\dots,m_n))=f^\mathcal{N}(h(m_1),\dots,h(m_n))$ for any $m_1,\dots,m_n\in M$,
\item $h(R^\mathcal{M}(m_1,\dots,m_n))\Rightarrow R^\mathcal{N}(h(m_1),\dots,h(m_n))$ for any $m_1,\dots,m_n\in M$,
\item $h(c^\mathcal{M})=c^\mathcal{N}$ for any constant symbol $c\in \mathcal{L}$.
\end{itemize}
\end{defi}

\begin{defi}
An $\mathcal{L}$-term is a sequence of constant symbols, function symbols, and variables $x_1,x_2,\dots$ that is defined recursively as follows:
\begin{itemize}
\item All constant symbols and variables are $\mathcal{L}$-terms.
\item If $t_1,\dots,t_n$ are $\mathcal{L}$-terms and $f$ is a $n_f$-ary function symbol, then $f(t_1,\dots,t_n)$ is an $\mathcal{L}$-term.
\end{itemize}
\end{defi}

\begin{defi}\label{defi:formula}
An $\mathcal{L}$-formula is a sequence of $\mathcal{L}$-terms,~$=,\neg,\wedge$, and $\exists$ that is defined recursively as follows:
\begin{enumerate}
\def\theenumi{\hbox{\hfil{$\mathrm{\roman{enumi}}$}\hfil}}
\item If $t_1$ and $t_2$ are $\mathcal{L}$-terms, then $t_1=t_2$ is an $\mathcal{L}$-formula.
\item If $t_1,\dots,t_n$ are $\mathcal{L}$-terms and $R$ is an $n$-ary relation symbol, then $R(t_1,\dots,t_n)$ is an $\mathcal{L}$-formula.
\item If $\Psi_1$ and $\Psi_2$ are $\mathcal{L}$-formulas, then $\Psi_1\wedge \Psi_2$ is an $\mathcal{L}$-formula.
\item If $\Psi$ is an $\mathcal{L}$-formula and $x$ is a variable, $\exists x \Psi$ is an $\mathcal{L}$-formula.
\item If $\Psi$ is an $\mathcal{L}$-formula, then $\neg \Psi$ is an $\mathcal{L}$-formula.
\end{enumerate}
\end{defi}

The following abbreviations are often used:
\begin{itemize}
\item $\Psi_1\lor \Psi_2$ denotes $\neg(\neg\Psi_1\wedge\neg\Psi_2)$.
\item $\Psi_1\to \Psi_2$ denotes $\neg(\Psi_1\wedge\neg \Psi_2)$.
\item $\Psi_1\leftrightarrow \Psi_2$ denotes $(\Psi_1\to \Psi_2)\wedge(\Psi_2\to \Psi_1)$.
\item $\forall x \Psi$ denotes $\neg (\exists x\neg \Psi)$.
\end{itemize}

\begin{defi}[Negation-free formula]\label{def:nag-free-formula}
An $\mathcal{L}$-formula is called \textit{negation-free}\footnote{Note that our `negation-free formulas' do not admit the quantifier $\forall$. } if it is consisted of $\mathcal{L}$-terms,~$=,\wedge,\lor$, and $\exists$.
\end{defi}

A variable $x$ is said to be {\it free} if it does not occur within the scope of a quantifier $\exists x$.
If an $\mathcal{L}$-formula $\varphi$ contains free variables $x_1,x_2,\dots,x_n$, we often denote it by $\varphi(x_1,\dots,x_n)$.

\begin{defi}
For an $\mathcal{L}$-structure $\mathcal{M}$, an $\mathcal{L}$-formula $\varphi(x_1,\dots,x_n)$, and an element $(m_1,\dots,m_n)\in M^n$, we define 
\[
\mathcal{M}\models \varphi(m_1,\dots,m_n)
\]
recursively as follows:
\begin{itemize}
\item If $t_1^\mathcal{M}(m_1,\dots,m_n)=t_2^\mathcal{M}(m_1,\dots,m_n)$, then $\mathcal{M}\models (t_1=t_2)(m_1,\dots,m_n)$.
\item If $R^\mathcal{M}(t_1^\mathcal{M}(m_1,\dots,m_n),\dots,t_l^\mathcal{M}(m_1,\dots,m_n))$, then\\ 
$\mathcal{M}\models (R(t_1,\dots,t_l))(m_1,\dots,m_n)$.
\item If both $\mathcal{M}\models\Psi_1(m_1,\dots,m_n)$ and $\mathcal{M}\models\Psi_2(m_1,\dots,m_n)$ are satisfied, then $\mathcal{M}\models(\Psi_1\wedge\Psi_2)(m_1,\dots,m_n)$.
\item If there exists $a\in M$ with $\mathcal{M}\models \Psi(m_1,\dots,a,\dots,m_n)$, then\\
$\mathcal{M}\models \exists x\Psi(m_1,\dots,x,\dots,m_n)$.
\item If $\mathcal{M}\not\models\Psi(m_1,\dots,m_n)$, then $\mathcal{M}\models\neg\Psi(m_1,\dots,m_n)$.
\end{itemize}
\end{defi}
If $\mathcal{M}\models \Psi(m_1,\dots,m_n)$, we say ``$\Psi(m_1,\dots,m_n)$ is {\it true} over $\mathcal{M}$.''

\subsection{Proof of Lemma \ref{prop:inherit}.}\label{sec:appA.2}

Let $\mathcal{L}$ be a language, and  $\mathcal{M},\o{\mathcal{M}}$ be $\mathcal{L}$-structures.
We let $M$ and $\o{M}$ denote the domain of $\mathcal{M}$ and $\o{\mathcal{M}}$, respectively.
Consider a morphism $\mathcal{M}\to \o{\mathcal{M}}$ of $\mathcal{L}$-structures.

Assume $\mathcal{M}\models \psi(a_1,\dots,a_n)$ for a negation-free $\mathcal{L}$-formula $\psi(x_1,\dots,x_n)$ and $(a_1,\dots,a_n)\in M^n$.
We prove $\o{\mathcal{M}}\models \o{\psi}(\o{a_1},\dots,\o{a_n})$ by induction (See Definitions \ref{defi:formula} and \ref{def:nag-free-formula}).
First, for $t_1=t_2$, we have
\begin{align*}
&\mathcal{M}\models (t_1=t_2)(a_1,\dots,a_n) \\
&\Rightarrow
t_1^\mathcal{M}(a_1,\dots,a_m)=t_2^\mathcal{M}(a_1,\dots,a_m)\\
&\Rightarrow
t_1^{\o{\mathcal{M}}}(\o{a_1},\dots,\o{a_m})=t_2^{\o{\mathcal{M}}}(\o{a_1},\dots,\o{a_m})\qquad (\because \mbox{$x\mapsto \overline{x}$ is a $\mathcal{L}$-morphism})\\
&\Rightarrow
\o{\mathcal{M}}\models (\o{t_1}=\o{t_2})(\o{a_1},\dots,\o{a_n}).
\end{align*}
The same argument works in the case of $R(t_1,\dots,t_n)$.
Next assume that the assertion holds for $\Psi_1(x_1,\dots,x_n)$ and $\Psi_2(x_1,\dots,x_n)$.
Then, we have
\begin{align*}
\mathcal{M}\models (\Psi_1\wedge \Psi_2)(a_1,\dots,a_n) & \Rightarrow
\mathcal{M}\models \Psi_1(a_1,\dots,a_n)\quad \mbox{and}\quad \mathcal{M}\models  \Psi_2(a_1,\dots,a_n)\\
& \Rightarrow
\o{\mathcal{M}}\models \o{\Psi_1}(\o{a_1},\dots,\o{a_n})\quad \mbox{and}\quad \o{\mathcal{M}}\models  \o{\Psi_2}(\o{a_1},\dots,\o{a_n})\\
& \Rightarrow
\o{\mathcal{M}}\models (\o{\Psi_1\wedge \Psi_2})(\o{a_1},\dots,\o{a_n}),
\end{align*}
and
\begin{align*}
\mathcal{M}\models (\Psi_1\lor \Psi_2)(a_1,\dots,a_n) 
& \Rightarrow
\mathcal{M}\not\models (\neg\Psi_1\wedge\neg\Psi_2)(a_1,\dots,a_n)\\
& \Rightarrow
\mathcal{M}\models \Psi_1(a_1,\dots,a_n)\quad \mbox{or}\quad \mathcal{M}\models  \Psi_2(a_1,\dots,a_n)\\
& \Rightarrow
\o{\mathcal{M}}\models \o{\Psi_1}(\o{a_1},\dots,\o{a_n})\quad \mbox{or}\quad \o{\mathcal{M}}\models  \o{\Psi_2}(\o{a_1},\dots,\o{a_n})\\
& \Rightarrow
\o{\mathcal{M}}\not\models (\neg\o{\Psi_1}\wedge \neg\o{\Psi_2})(\o{a_1},\dots,\o{a_n})\\
& \Rightarrow
\o{\mathcal{M}}\models (\o{\Psi_1\lor \Psi_2})(\o{a_1},\dots,\o{a_n}).
\end{align*}
Moreover, 
\begin{align*}
\mathcal{M}\models \exists x\Psi(x,a_2,\dots,a_n)
&\Rightarrow
\mbox{there exists some $a_1\in M$ with }\mathcal{M}\models \Psi(a_1,a_2,\dots,a_n)\\
&\Rightarrow
\mbox{there exists some $a_1\in M$ with }\o{\mathcal{M}}\models \o{\Psi}(\o{a_1},\o{a_2},\dots,\o{a_n})\\
&\Rightarrow
\mbox{there exists some $A\in \o{M}$ with }\o{\mathcal{M}}\models \o{\Psi}(A,\o{a_2},\dots,\o{a_n})\\
&\Rightarrow
\o{\mathcal{M}}\models \exists X\o{\Psi}(X,\o{a_2},\dots,\o{a_n}).
\end{align*}
Therefore, the assertion holds for any negation-free formula.

\section{Basics on the combinatorics of Young tableaux}\label{sec:appB}

A box $B$ in a Young diagram is said to be placed in a \textit{corner} if there exists no box below nor on the right to $B$.
For a skew diagram $\lambda/\mu$, a corner of $\lambda$ is called an {\it outside corner} and a corner of $\mu$ is called an {\it inside corner}.

A location at where no box exists is called a {\it hole}.
For a skew tableau $T$ and an inside corner $b$, the {\it jeu de taquin slide starting from $b$} is defined as follows:
(i) Compare the two entries in the boxes below and on the right to the hole $b$, and slide a box with smaller number to $b$.
If these two entries are same, slide the box below $b$.
(ii) Compare the two entries in the boxes below and on the right to the new hole, and slide a box according to the same rule in (i).
(iii) Repeat (ii) until the hole reaches to an outside corner.

The following is an example of a jeu de taquin.
Here the grayed boxes denote the hole.
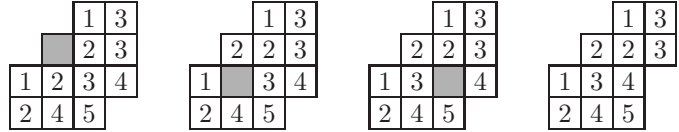
\[
\Tableau{ \bl& \bl& 1 & 3\\ \bl& \gray & 2 & 3\\ 1 & 2 & 3& 4\\ 2&4&5}
\qquad
\Tableau{ \bl& \bl& 1 & 3\\ \bl& 2 & 2 & 3\\ 1 & \gray  & 3&4\\ 2&4&5}
\qquad 
\Tableau{ \bl& \bl& 1 & 3\\ \bl& 2 & 2 & 3\\ 1 & 3  & \gray &4\\ 2&4&5}
\qquad 
\Tableau{ \bl& \bl& 1 & 3\\ \bl& 2 & 2 & 3\\ 1 & 3  & 4\\ 2&4&5}
\]
In this example, a jeu de taquin slide starts from the $2^\mathrm{nd}$ row, and ends at the $3^\mathrm{rd}$ row.

Let $T$ be a tableau and $t$ be a number.
The {\it row bumping} (or {\it row insertion}) of $t$ to $T$ is defined as follows:
(i) If $t$ is equal to or greater than all the entries in the $1^\mathrm{st}$ row of $T$, put a new box filled with $t$ at the end of this row.
If not, $t$ ``bumps'' the leftmost entry greater than $t$.
The bumped number proceeds to the next row.
(ii) Apply the same procedure as (i) to the next row and the bumped number.
(iii) Repeat (ii) until the bumped number is put at the end of some row.

Here is an example of a row bumping of 3 to a tableau.
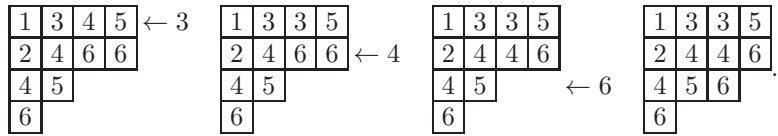
\[
\Tableau{1&3&4&5&\bl \quad\la 3\\2&4&6&6\\4&5\\6}
\qquad
\Tableau{1&3&3&5\\2&4&6&6&\bl \quad\la 4\\4&5\\6}
\qquad
\Tableau{1&3&3&5\\2&4&4&6\\4&5&\bl& \bl &\bl \quad\la 6\\6}
\qquad
\Tableau{1&3&3&5\\2&4&4&6\\4&5&6\\6}.
\]
The tableau obtained by the row bumping of $t$ to $T$ is denoted by 
\[
T\la t\qquad \mbox{or}\qquad T\la \Tableau{t}.
\]

\bibliographystyle{amsplain}
\bibliography{JDTen}

\providecommand{\bysame}{\leavevmode\hbox to3em{\hrulefill}\thinspace}
\providecommand{\MR}{\relax\ifhmode\unskip\space\fi MR }
\providecommand{\MRhref}[2]{%
  \href{http://www.ams.org/mathscinet-getitem?mr=#1}{#2}
}
\providecommand{\href}[2]{#2}
\begin{thebibliography}{10}

\bibitem{bernsteinRSK2001}
Arkady Berenstein and Anatol~N. Kirillov, \emph{The
  {R}obinson-{S}chensted-{K}nuth bijection, quantum matrices and piece-wise
  linear combinatorics}, Proceedings of 13th International Conference on Formal
  Power Series and Algebraic Combinatorics, Arizona State University, 2001.

\bibitem{fulton_1996}
William Fulton, \emph{Young tableaux: With applications to representation
  theory and geometry}, London Mathematical Society Student Texts, Cambridge
  University Press, 1996.

\bibitem{iwao2018Tropical}
Shinsuke Iwao, \emph{{Tropical integrable systems and {Y}oung tableaux: shape
  equivalence and {L}ittlewood-{R}ichardson correspondence}}, Journal of
  Integrable Systems \textbf{3} (2018), no.~1, xyy011.

\bibitem{iwao2018discrete}
Shinsuke Iwao and Hidetomo Nagai, \emph{The discrete toda equation revisited:
  dual $\beta$-grothendieck polynomials, ultradiscretization, and static
  solitons}, Journal of Physics A: Mathematical and Theoretical \textbf{51}
  (2018), no.~13, 134002.

\bibitem{kakei2015en}
Yosuke Katayama and Saburo Kakei, \emph{Jeu de taquin slide and ultradiscrete
  {KP} equation (in {J}apanese)}, Reports of RIAM Symposium \textbf{26AO-S2}
  (2015), 133--138.

\bibitem{2001phco.conf...82K}
Anatol~N. Kirillov, \emph{{Introduction to Tropical Combinatorics}}, Physics
  and Combinatorics (A.~N. {Kirillov}, A.~{Tsuchiya}, and H.~{Umemura}, eds.),
  April 2001, pp.~82--150.

\bibitem{marker2006model}
David Marker, \emph{Model theory: an introduction}, vol. 217, Springer Science
  \& Business Media, 2006.

\bibitem{mikami2012en}
Yu~Mikami, \emph{Relation between jeu de taquin slide and ultradiscrete {KP}
  equation (in {J}apanese)}, Master's thesis, Graduate School of Science Kobe
  University, 2006.

\bibitem{noumi2004}
Masatoshi Noumi and Yasuhiko Yamada, \emph{Tropical
  {R}obinson-{S}chensted-{K}nuth correspondence and birational {W}eyl group
  actions}, Representation theory of algebraic groups and quantum groups
  (T.~Shoji, M.~Kashiwara, N.~Kawanaka, G.~Lusztig, and K.~Shinoda, eds.),
  vol.~40, Soc.\,Japan,\,Tokyo, 2004, pp.~371--442.

\bibitem{takahashi1990soliton}
Daisuke Takahashi and Junkichi Satsuma, \emph{A soliton cellular automaton},
  Journal of the Physical Society of Japan \textbf{59} (1990), no.~10,
  3514--3519.

\bibitem{tent2012course}
Katrin Tent and Martin Ziegler, \emph{A course in model theory}, Lecture notes
  in logic, Cambridge University Press, 2012.

\end{thebibliography}

\end{document}